 \documentclass[a4paper,12pt]{article}

\usepackage{amsthm,amsmath,hyperref,geometry,color,stmaryrd,bbm}
\usepackage[utf8]{inputenc}
\usepackage{amssymb}
\usepackage[english]{babel}
\usepackage{graphicx}
\usepackage{amsfonts,amssymb}
\usepackage{verbatim}
\usepackage{enumitem}

\newtheorem{thm}{Theorem}
\newtheorem{prop}[thm]{Proposition}
\newtheorem{rmq}[thm]{Remark}
\newtheorem{lem}[thm]{Lemma}
\newtheorem{cor}[thm]{Corollary}

\newtheorem*{defi*}{Definition}

\newtheorem*{assu*}{Assumption}

\newcommand{\po}{\left(}
\newcommand{\pf}{\right)}
\newcommand{\co}{\left[}
\newcommand{\cf}{\right]}
\newcommand{\cco}{\llbracket}
\newcommand{\ccf}{\rrbracket}
\newcommand{\R}{\mathbb R}
\newcommand{\N}{\mathbb N}
\newcommand{\W}{\mathcal W}

\newcommand{\dd}{\mathrm{d}}
\newcommand{\na}{\nabla}
\newcommand{\1}{\mathbbm{1}}
\newcommand{\Id}{I_d}
\newcommand{\bx}{\bar x}
\newcommand{\bv}{\bar v}

\newcommand{\nv}[1]{#1}
\newcommand{\err}[1]{#1}
\renewcommand{\phi}{\varphi}

\newcommand{\indep}{\perp \!\!\! \perp}
\newcommand{\refsmooth}{$\mathbf{(\na Lip)}$}
\newcommand{\refconvex}{$\mathbf{(Conv)}$}
\newcommand{\refordren}{$\mathbf{(\na^2 pol(}\ell\mathbf{))}$}
\newcommand{\refind}{$\mathbf{(\indep)}$}

\title{High-dimensional MCMC with a standard splitting scheme for the underdamped Langevin diffusion.}
 \author{Pierre Monmarch\'e}

\begin{document}

\maketitle

\begin{abstract}
The efficiency of a Markov sampler based on the underdamped Langevin diffusion is studied for  high dimensional targets with convex and smooth potentials. We consider  a classical second-order integrator which requires only one gradient computation per iteration. Contrary to previous works on similar samplers, a dimension-free contraction of Wasserstein distances and convergence rate for the total variance distance are proven for the discrete time chain itself.  Non-asymptotic Wasserstein and total variation efficiency bounds and concentration inequalities are obtained for both the Metropolis adjusted and unadjusted chains. \nv{In particular, for the unadjusted chain,} in terms of the dimension  $d$  and  the desired accuracy $\varepsilon$, the Wasserstein efficiency bounds are of order $\sqrt d / \varepsilon$ in the general case, $\sqrt{d/\varepsilon}$ if the Hessian of the potential is Lipschitz, and $d^{1/4}/\sqrt\varepsilon$ in the case of a separable target, in accordance with known results for other kinetic Langevin or HMC schemes.
\end{abstract}

%

\section{Introduction}

The Langevin diffusion (also called underdamped or kinetic Langevin diffusion) is the Markov process on $\R^d\times \R^d$ that solves the SDE
\begin{equation}\label{Eq:Continu-Langevin}
\left\{ \begin{array}{rcl}
\dd X_t & = & V_t \dd t \\
\dd V_t & = & -\na U(X_t) \dd t - \gamma V_t \nv{\dd t} + \sqrt{2\gamma} \dd B_t
\end{array}\right. 
\end{equation}
where $B$ is a standard $d$-dimensional Brownian motion, $U\in\mathcal C^2(\R^d)$ is called the potential (or log-likelihood) and $\gamma>0$ is a friction (or damping) parameter. Under general assumptions, it is known to be ergodic with respect to  the Gibbs measure with Hamiltonian $H(x,v) = U(x) +|v|^2/2$, namely the probability measure $\pi$  on $\R^{2d}$ with density proportional to $\exp(-H(x,v))$. It is a central subject of study in a wide variety of domains, at the intersection of probabilities, PDE and statistical physics, modeling e.g. plasmas or stellar systems. In this work we focus on the use of the Langevin diffusion in    Markov Chain Monte Carlo  (MCMC) algorithms in order to estimate expectations with respect to $\pi$.

 This process has been used for decades in molecular dynamics (MD), in particular due to physical motivations related to the motion of particles in classical physics (see  \cite{HOROWITZ,OldLangevin1,OldLangevin2,Leimkuhler,LelievreStoltz,LelievreFreeEnergy,SLSCC,BussiParrinello,Bourabee1,Bourabee2} and references within). In this context, the objective is not necessarily to compute only statistical averages, but also dynamical properties, like diffusion coefficients, transition times, reaction paths, etc., so that it is crucial to sample a trajectory of (a discrete-time approximation of) the Langevin diffusion, and not of another process with the same statistical equilibrium. More recently, the Langevin diffusion has gained some interest in the computational statistics and learning community \cite{Chatterji1,Chatterji2,Chatterji3,dalalyan1,Zajic}, in which case there is no preferred underlying dynamics and only statistical averages with respect to $\pi$ matter. The present work follows this viewpoint. Besides, in this context,  the original question is to compute expectations with respect to $\pi_1\propto \exp(-U)$ the first $d$-dimensional marginal of $\pi$, and the velocity $V$ is added as an auxiliary variable in order to enhance the sampling. This is called a lifted MCMC method.
 
 We will focus on the case where $U$ is $m$-convex and $L$-smooth (see Assumptions~\refsmooth\ and \refconvex\ below).
 These are usual conditions in the statistics community to compare MCMC methods, the objective being then to obtain non-asymptotic explicit estimates in term of the dimension $d$, see e.g. \cite{durmus2019,Chatterji1,dalalyan1}. These conditions are far from natural for MD applications where, except for very particular harmonic toy models, the potentials are highly non-convex and singular due to strong short-range repulsion of  nuclei. On the other hand, and contrary to the recent works in computational statistics, we will study a discretization scheme of  \eqref{Eq:Continu-Langevin}  obtained by a splitting procedure, as commonly used in MD  \cite{BussiParrinello,Leimkuhler,Bourabee1,Bourabee2,SLSCC} and implemented in all MD codes. 
 
The study for MCMC purpose of a discretization of the Langevin diffusion for a convex and smooth potential has recently been conducted in \cite{Chatterji1,dalalyan1,Chatterji3,Zajic}. Concerning Metropolis adjusted chains, some non-asymptotic results have recently been established in  \cite{Dwivedi2,Dwivedi,Mangoubi,Mangoubi2,EberleHMC} for  HMC  samplers or other Metropolis-adjusted algorithms.   Let us give an informal summary of our contribution with respect to these works (a more detailed and quantitative discussion is provided in Section~\ref{Sec:otherworks} after our results have been rigorously stated).
 
 \begin{itemize}
 \item As far as non-asymptotic efficiency bounds are concerned, standard splitting discretisation  schemes (such as the OBABO chain studied in the present work, see  Section~\ref{Sec:OBABO}) had not been studied yet. Contrary to the first order scheme of \cite{Chatterji1,dalalyan1,Chatterji2,Zajic}, the OBABO sampler is a second order scheme (in the time-step). Like classical schemes, it only requires one computation of $\na U$ per step, contrary to the second order scheme of \cite{dalalyan1,dalalyan2} that also requires a computation of $\na^2 U$.  It belongs to the class of splitting schemes considered in \cite{Leimkuhler}, but \cite{Leimkuhler} only provides a formal expansion of the invariant measure in the time-step, and only for two particular splitting (and not the one studied here). A second order scheme with the same computational complexity is also introduced in \cite{Chatterji3}, and we obtain similar results in term of efficiency.   Contrary to the schemes studied in \cite{Chatterji3,Leimkuhler}, the OBABO  scheme is naturally connected to HMC type samplers, and in particular the discretization bias at equilibrium is only due to a deterministic approximation of the Hamiltonian dynamics by a Verlet (or leapfrog) scheme, which has some nice consequences in the analysis.
 \item Contrary to the previous works on other schemes based on the underdamped Langevin diffusion, instead of establishing long-time convergence estimates for the continuous-time process and then conducting an analysis of the discretization error,  we directly prove a Wasserstein contraction for the discrete time Markov chain itself. This classically yields nice concentration results (and thus non-asymptotic confidence intervals) for empirical averages\nv{, see Theorem~\ref{thm_main:concentration}, which is not the case in previous works}. Then, we study the bias at equilibrium separately. Notice that, ultimately, gathering these two parts yields efficiency bounds that are similar to those obtained with the method that relies only on the convergence rate of the continuous-time process, nevertheless this work was initially motivated by the theoretical question considered in \cite{QinHobert} of obtaining convergence rates for (discrete-time) MCMC samplers in high dimension. \nv{This issue is not addressed with the basic method. Finally, with our method, we obtain that empirical averages of the discrete-time chain have a (biased) long-time limit (again, this is Theorem~\ref{thm_main:concentration}), which is a first step for using Romberg interpolation techniques, see e.g. \cite{DurmusRomberg,Bach} (this issue is postponed to a future work, see \cite{MoiRomberg}).}
 \item We establish a Wasserstein/total variation regularization property. To our knowledge, this is the first result of this kind for a Markov chain based on an hypoelliptic non-elliptic diffusion. As a consequence, contrary to \cite{Chatterji1,dalalyan1,Zajic}, we obtain results for the total variation distance. However, we acknowledge that schemes that are based on the exact solving of the SDE~\eqref{Eq:Continu-Langevin} but with $\na U(X_t)$ replaced by $\na U(X_{\delta \lfloor t/\delta\rceil})$  where $\delta>0$ is the time-step (which is the case in \cite{Chatterji1,dalalyan1,Chatterji3,Zajic} and not of our scheme) have Gaussian transitions and thus enjoy a similar  regularization property, although it hasn't been stated in previous works. In \cite{Chatterji3}, bounds are obtained for the Kullback-Leibler divergence, which is stronger, but with hypocoercive PDE techniques, while we use an elementary coupling approach. Adapting their method to our case is not straightforward, as one iteration of our scheme is not based on solving \eqref{Eq:Continu-Langevin} with fixed forces.
 \item In order to see that second order schemes are more efficient than first order ones, the gradient Lipschitz assumption is not sufficient, some information is required on higher order derivatives of $U$. For this reason, in \cite{dalalyan1,Chatterji3}, the assumption that $\na^2 U$ is Lipschitz is enforced. We work under a weaker condition (see Assumption~\refordren\ below), that holds for instance if $\|\na^{(k)} U\|_\infty < +\infty$ for some $k>2$. When $\na^2 U$ is Lipschitz, we obtain efficiency bounds similar to \cite{dalalyan1,Chatterji3}.
 \item We establish non-asymptotic efficiency bounds and confidence intervals for the Metropolis-adjusted OBABO scheme, which is an HMC sampler with a particular choice of parameters. This is the first results of this kind for a Metropolis-adjusted sampler based on the (kinetic) Langevin equation. By comparison, the recent works  \cite{Dwivedi2,Mangoubi,Mangoubi2,EberleHMC} are concerned with the classical HMC sampler, with complete refreshment of velocities separated by long run of the Hamiltonian dynamics (and do not provide confidence intervals). Also, contrary to \cite{Dwivedi2,Dwivedi,Mangoubi,Mangoubi2}, we do not study the case of so-called warm start for the initial distribution, and we obtain non-asymptotic bounds that involve only some moments of the initial distribution (similarly to the unadjusted case). 
 \end{itemize}

 We are mainly interested in the dependency in the dimension $d$, the time-step $\delta$ and the accuracy $\varepsilon$ of the explicit estimates. \nv{Concerning the dependency on other parameters like $m$ or $L$, in practice, all our estimates could be improved by considering a preconditioned algorithm as in \cite{Zajic} and/or by taking into account explicitly a Gaussian part of $U$, i.e. decomposing $U(x) = x\cdot S^{-1}x/2 + \tilde U(x)$ with some covariance matrix $S$ and potential $\tilde U$, see e.g. \cite{Pavliotis}}.

\medskip

The article is organized as follows. The OBABO sampler and its Metropolis-adjusted counterpart are presented in Section~\ref{Sec:definitions}, together with some general considerations on Wasserstein distances. Our main results are stated in Section~\ref{Sec:mainresults}, decomposed as: results on the long-time behavior of the OBABO chain (Section~\ref{Sec:mainresOBABO}); on the equilibrium bias (Section~\ref{Sec:mainresBias}); on the Metropolis-adjusted sampler (Section~\ref{Sec:mainresMOBABO}). The results are discussed and related to other works in Section~\ref{Sec:otherworks}. Section~\ref{Sec:proofOBABO} contains the proofs for the long-time behaviour of the OBABO sampler, Section~\ref{Sec:CoupleVerlet} the proofs for the equilibrium bias and Section~\ref{Sec:proofMOBABO} the proofs for the Metropolis-adjusted chain.

\section{Definitions}\label{Sec:definitions}

\subsection{A second-order scheme for the Langevin diffusion}\label{Sec:OBABO}

We present the Markov chain which is the main object of study of this work, obtained as a time-discretization of the Langevin diffusion (already considered by Horowitz in \cite{HOROWITZ}, see also \cite{BussiParrinello} and references within). First, we simply define the transition of the Markov chain, and the rest of the section is devoted to the motivation and discussion of this definition.  Let $\delta,\gamma>0$ be respectively the time-step and friction coefficient, and denote $\eta = e^{-\delta\gamma/2}$. We consider the \nv{time-homogeneous} Markov chain $(x_n,v_n)_{n\in\N}$ on $\R^d\times\R^d$ with transitions given by
\begin{equation}
\left\{\begin{array}{rcl}
x_{1} & = & x_0+ \delta \po \eta v_0 + \sqrt{1-\eta^2} G\pf  - \frac{\delta^2}2 \na U(x_0)\\
v_1 & = & \eta^2v_0 - \frac{\delta\eta }{2}\po \na U(x_0) + \na U(x_1)\pf + \sqrt{1-\eta^2} \po \eta G + G'\pf\,,
\end{array}\right.\label{Eq:ABOBA}
\end{equation}
where $G$ and $G'$ are two independent standard (mean $0$ variance $\Id$)  $d$-dimensional Gaussian random variables. Denote by $P$ the Markov transition operator associated to the chain, i.e. $P\varphi(x_0,v_0) = \mathbb E(\varphi(x_1,v_1))$ for all measurable bounded $\varphi$.

A first very important remark is that, although it seems that in one iteration, $\na U$ has to be computed for two different values of $x$, this is in fact only true for the first iteration, afterward $\na U(x_n)$ is already known from the previous iteration and only $\na U(x_{n+1})$ is computed. A second remark is that, contrary to the schemes considered in e.g. \cite{Chatterji1,dalalyan1,Chatterji3,Zajic} (where the transition is obtained by following for a time $\delta$ the continuous-time process \eqref{Eq:Continu-Langevin} but with constant forces $\na U(x_t) = \na U(x_0)$), the law of  $(x_{1},v_{1})$  conditionally to $(x_0,v_0)$ is not a Gaussian distribution, because of the term $\na U(x_1)$ in $v_1$ (except in the particular case where $U$ is quadratic).

The scheme makes more sense when it is decomposed as the following successive steps:
\begin{align*}
v_0'\  & =  \ \eta v_0 + \sqrt{1-\eta^2} G\tag{O}\\
v_{1/2}\ &=  \  v_0' - \delta/2 \na U(x_0) \tag{B}\\
x_1 \ &= \  x_0 + \delta v_{1/2} \tag{A}\\
v_1'\  & = \ v_{1/2} - \delta/2 \na U(x_1)\tag{B}\\
v_1 \ & = \  \eta v_1' + \sqrt{1-\eta^2} G'\,.\tag{O}
\end{align*}
Here,  we use the notations O, \nv{B and A} of \cite{Leimkuhler}, referring respectively to the damping (or friction/dissipation, or partial resampling), the acceleration and the free transport parts of the dynamics.  So, in the rest of the article we will refer to the Markov chain $(x_n,v_n)_{n\in\N}$ as  the OBABO sampler which, contrary to some of its other names like second order Langevin or midpoint Euler-Verlet-Midpoint Euler \cite{BussiParrinello,HOROWITZ,LelievreFreeEnergy}, has the advantage to give a compact yet explicit description of the specific scheme (in particular by comparison with other second order schemes for the underdamped Langevin diffusion).

Note that the parts BAB of the scheme is the classical velocity Verlet algorithm for the Hamiltonian dynamics. More generally, the Verlet, OBABO, or BAOAB (from \cite{Leimkuhler}) algorithms all shares a common palindromic form since they are obtained from a Strang/Trotter splitting of the generator of the Langevin diffusion. We now give a brief and informal presentation of the latter, and refer to \cite{Leimkuhler,LelievreFreeEnergy} or \cite[Section 5]{Monmarche2019Kinetic_walk} for more details. The transition semigroup $P_t$ associated to \eqref{Eq:Continu-Langevin} is informally of the form $e^{t\mathcal L}$ where the infinitesimal generator $\mathcal L$ can be decomposed as $\mathcal L = \mathcal L_A+\mathcal L_B +  \mathcal L_O$ with
\[\mathcal L_A \ = \  v\cdot \na_x\,\qquad \mathcal L_B \ = \ -\na U(x) \cdot \na_v \,\qquad \mathcal L_O \ = \  -\gamma v\cdot \na_v + \gamma  \Delta_v\,.\]
Using that $e^{\delta(a+b)} = e^{\delta/2a}e^{\delta b}e^{\delta/2a} + o(\delta^2)$, we get formally a second order approximation of $P_\delta$ by
\[P_\delta \ = \ e^{\delta/2 \mathcal L_O}e^{\delta/2 \mathcal L_B}e^{\delta  \mathcal L_A}e^{\delta/2 \mathcal L_B}e^{\delta/2 \mathcal L_O} + o(\delta^2)\,.\]
The operator on the right-hand side is exactly the transition kernel of the OBABO sampler. At least formally, the fact that the transition kernel of the chain approximates the transition semi-group of the continuous process up to second order in $\delta$ can be shown to yield a similar order for the error on the invariant measure, see \cite{Leimkuhler,Monmarche2019Kinetic_walk} for   general discussions \nv{and more details} and  Proposition~\ref{prop_main:biais} below for a rigorous and quantitative statement in the particular OBABO case. \nv{More precisely, in  Proposition~\ref{prop_main:biais}, we get an error term on the invariant measure of order $\delta^2$ for the Wasserstein and total variation distances, to compare to e.g. \cite[Theorem 2]{dalalyan1} or the proof of   \cite[Theorem 1]{Chatterji1} where the error is of order $\delta$ for the Wasserstein distance (see Section~\ref{Sec:otherworks} for the consequence of this on the efficiency bounds).} For convex and smooth potentials, for a suitable choice of $\gamma$, the continuous-time process \eqref{Eq:Continu-Langevin} converges to equilibrium at a rate that is independent of the dimension, so that the error estimates obtained in \cite{Chatterji1,dalalyan1,Chatterji3,Zajic} depends on the dimension only because of the error on the invariant measure (which requires the time-step to be small enough). By replacing a first order discretization scheme by a second order one, we expect to improve the estimates accordingly, which is indeed what is observed in \cite{dalalyan1,Chatterji3} where other second-order schemes are considered, \nv{and in the present work (again, we refer to Section~\ref{Sec:otherworks})}.

The reason  we chose to study the OBABO algorithm rather than the BAOAB one or other similar scheme is the following. The OBABO scheme contains at its core a (deterministic) Verlet part BAB, which is time reversible (in the physicist sense, i.e. up to a reflection of the velocity). This fact classically leads to a natural Metropolized version of the algorithm, of Hamiltonian (or Hybrid) Monte Carlo (HMC) type, presented below.  This Metropolis-adjusted algorithm is interesting in itself, and we use it as an auxiliary tool in the study of the bias of the OBABO chain. Besides, an unintended benefit of our choice of integrator is that, for the OBABO scheme, establishing a Wasserstein/total variation regularization property (Proposition~\ref{prop_main:W/TV}) is rather straightforward because of the particular location of the noise in the scheme. That being said, it should be possible to adapt most (if not all) of our arguments to the BAOAB or other similar second order schemes like in \cite{Chatterji3}. 
 Finally, remark that the BAOAB scheme is promoted in \cite{Leimkuhler} because, in the overdamped limit $\gamma\rightarrow +\infty$, the invariant measure is correctly sampled at fourth order in the time-step. However, as $\gamma\rightarrow +\infty$, the convergence rate of the chain goes to zero, so it is not clear that this argument is decisive in the choice of the integrator (after a suitable scaling of time, the BAOAB converges as $\gamma\rightarrow +\infty$ to a discretization scheme of the overdamped Langevin diffusion which has a correct equilibrium up to second order).

Recall the goal of the algorithm is to compute averages with respect to $\pi_1$  the first $d$-dimensional marginal of $\pi$. As such, we are not constrained for the equilibrium of the velocities,  and typically we could chose a symmetric Gaussian distribution with variance different from $1$. Considering the Hamiltonian $H(x,v) = U(x)+|v|^2/(2\sigma^2)$, we would have a new parameter $\sigma$ to optimize, which could improve the results.  Nevertheless, in that case, we can always chose $\sigma= 1$ by the change of variables $v\leftarrow v/\sigma$, $\delta\leftarrow \delta \sigma$ and $\gamma \leftarrow \gamma/\sigma$. As a consequence, without loss of generality, we take $\sigma= 1$.

 As mentioned above, the particular structure of the scheme has some nice implications for the theoretical study. Of course, in practice, if we are only interested in computing averages that only involve the position $x_n$, then the scheme is equivalent to alternate a Verlet step BAB and a partial refreshment O$^2$ (which is simply O but with a full rather than half time-step, i.e. $\eta$ is replaced by $e^{-\delta\gamma}$), since $($OBABO$)^n=$O$($BABO$^2)^{n-1}$BABO and the first and last O have no effect on the position.
 
 Finally, let us mention that our analysis can be straightforwardly adapted to get results on the O(BAB)$^k$O scheme for some fixed $k\in \N_*$ independent from $\delta$ (i.e. $k$ Verlet steps are performed between the partial refreshments). Indeed, as $\delta$ vanishes, the corresponding chain still converges toward the continuous time process \eqref{Eq:Continu-Langevin}, which is the main ingredient of our results concerning long-time convergence, and similarly there is no particular difficulty to adapt the results concerning discretization errors. For fixed parameters $k$ and $\gamma$, this extension does not yield any improvement in terms of scaling in $\delta$ and $d$ (also, note that one iteration of  O(BAB)$^k$O requires $k$ computations of $\na U$) and thus we don't consider it to avoid an unnecessary inflation of notations.
 
\subsection{Metropolis-adjusted algorithm}\label{Sec:MOBABO}

The O parts of the scheme leaves invariant the standard Gaussian distribution for the velocities, hence the target $\pi$. This is not the case of the Verlet part BAB because of the numerical approximation. As a consequence, $\pi$ is not invariant for the OBABO chain. Nevertheless, this is classically fixed by adding a Metropolis  accept/reject step on this part of the dynamics, which gives an HMC algorithm \cite{HOROWITZ,Neal}. We now detail this. In the following we denote
\[\Phi_V(x,v) \ = \ \po x+ \delta v - \frac12 \delta^2 \na U(x) , v-\frac12 \delta \po \na U(x) + \na U\po x+ \delta v - \frac12 \delta^2 \na U(x)\pf \pf\pf\,,\]
so that the steps BAB  read  $(x_{1},v_{1}') = \Phi_V(x_0,v_0')$. Denote also $\Phi_R(x,v) = (x,-v)$. It is then straightforward to check that the Verlet algorithm is time-reversible, i.e. $\Phi_R \Phi_V\Phi_R\Phi_V(x,v)=(x,v)$ for all $x,v\in\R^d$. As a consequence, the deterministic proposal kernel
\[q\po (x,v),\dd y \dd w\pf \ = \ \delta_{\Phi_R\Phi_V(x,v)}(\dd y \dd w)\]
is symmetric. Denote $P_{MH}$ the Markov transition  operator of the Metropolis-Hastings algorithm with proposal kernel $q$. A transition of the corresponding chain is given by:
\[\begin{array}{rcll}
(x',v')  & = & \Phi_V(x_0,v_0) & \\ 
(x_1,v_1)  & = & (x',-v') & \text{if } W\leqslant \alpha(x_0,v_0):=\exp\po - \po H\po \Phi_V(x_0,v_0)\pf - H(x_0,v_0)\pf_+\pf\\
& =  & (x_0,v_0) & \text{otherwise,}
\end{array}\]
where $W$ is uniformly distributed over $[0,1]$. We used that $H(x',v')=H(x',-v')$. By construction of the Metropolis-Hastings algorithm, $P_{MH}$ is reversible (in the probabilistic sense) with respect to $\pi$. Moreover, $\pi$ is also invariant by the transformation $v\leftrightarrow -v$, namely $\pi P_R = \pi$ where $P_R\varphi(x,v)=\varphi(x,-v)$. Denote $P_{MV} = P_R P_{MH}$, which is similar to a Verlet transition BAB of the previous section except for the accept/reject step. Remark that, in case of rejection, the velocity is reflected. This is invisible in the classical HMC algorithm for which the velocity is fully resampled after each  accept/reject step (which will not be the case here).

Denote by $P_{O}$  the Markov transition kernel on $\R^d\times\R^d$ corresponding to the O transition, i.e.
\begin{eqnarray*}
(x_{1},v_{1}) &=& \po x_0, \eta  v_0 + \sqrt{1-\eta^2} G\pf \qquad \text{with}\qquad G \sim \mathcal N\po 0,\Id\pf\,.
\end{eqnarray*}
In other words, $P_O$ is given for any bounded observable $\varphi$ by
\[P_O \varphi (x,v) \ = \ (2\pi)^{-d/2}\int_{\R^d} \varphi\po x, \eta v + \sqrt{1-\eta^2} w\pf e^{-|w|^2/2} \dd w\,.\]
The Metropolis-adjusted version of the OBABO algorithm, introduced in \cite{HOROWITZ}, is the Markov chain with transition operator
\[P_M \ = \ P_O P_{MV} P_O\,.\]
By construction, $\pi$ is invariant for $P_M$. Although $\pi$ is reversible for $P_O$, $P_R$ and $P_{MH}$, this property is not conserved by composition and $\pi$ is not reversible with respect to $P_M$. Besides, using that $P_RP_OP_R = P_O$ and $P_R^2 = Id$, we see that that   the adjoint operator $P_M^* = P_O P_{MH}P_RP_O$ in $L^2(\pi)$ is given by  $P_M^*  = P_R P_M P_R$, i.e. the process is reversible in the physics sense, up to a reflection of the velocity.  We won't use this property in the article but we note that it allows  to get an $L^2$ spectral gap from a contraction in an other distance, like a Wasserstein distance or the $H^1$ norm, as in the reversible case, see e.g. \cite{DoucetHMC}.

The operator $P_M$ corresponds to the following transition:
\[\begin{array}{rcll}
v_0'   & = &   \eta v_0 + \sqrt{1-\eta^2} G &  \\
v_{1/2}  &= &    v_0' - \delta/2 \na U(x_0)  &  \\
\tilde x_1   &= &   x_0 + \delta v_{1/2}  & \\
\tilde v_1  & = &  v_{1/2} - \delta/2 \na U(\tilde x_1) & \\
(x_1,v_1') &= & (\tilde x_1,\tilde v_1) & \text{if }W\leqslant \alpha(x_0,v_0')\\
& =& (x_0,-v_0') & \text{otherwise}\\
v_1 &=& \eta v_1' + \sqrt{1-\eta^2} G'\,. &
\end{array}\]
Remark that, in case of acceptation (namely if $W\leqslant \alpha(x_0,v_0')$) then this is exactly the OBABO transition. Similar Metropolis-adjusted discretizations of the Langevin diffusion have been studied in  \cite{SLSCC,Bourabee1,Bourabee2,Pavliotis}, see also references within. 

Due to the particular scaling of the partial refreshment mechanism, this HMC chain converges as $\delta$ vanishes to the Langevin diffusion rather than to the continuous-time Randomized HMC (see \cite{DoucetHMC} and references within) where dissipativity in the velocity is ensured by a (possibly partial) resampling of the velocities at constant rate, rather than an Ornstein-Uhlenbeck diffusion.

Following motivations similar to the unadjusted case, we call this specific Metropolis-adjusted splitting scheme the OM(BAB)O chain (the M($\cdot$) standing for Metropolis).

\subsection{Wasserstein distances}

Let $\rho$ be  a distance on $\R^d$. For $p\in[1,\infty)$, denote by $\mathcal P_{p,\rho}(\R^{d})$ the set of probability distributions on $\R^d$ having a finite $p$-th moment for $\rho$, i.e. $\nu\in\mathcal P_{p,\rho}(\R^d)$ if
\[\int_{\R^d} \rho^p(0,x) \nu(\dd x) \ < \ +\infty\,.\]
For $p\in[1,\infty]$ the Wasserstein distance $\W_{p,\rho}$ on $\mathcal P_{p,\rho}(\R^{d})$ is given for $\nu,\mu\in\mathcal P_{p,\rho}(\R^d)$ by
\[\W_{p,\rho} \po \nu,\mu\pf \ = \ \inf_{\eta\in \Gamma(\nu,\mu)}   \| \rho \|_{L^p(\eta)}\,, \]
where $\Gamma(\nu,\mu)$ is the set of transference plan of $\nu$ and $\mu$, namely is the set of probability distributions on $\R^d\times\R^d$ with first marginal $\nu$ and second marginal $\mu$. In other words,  
\[\W_{p,\rho}^p \po \nu,\mu\pf \ = \ \inf_{X\sim \nu,Y\sim\mu} \mathbb E\po  \rho^p(X,Y)  \pf\,, \]
although this is only a formal definition since the infimum is not taken on a proper set. If $(X,Y)$ is a random variable with law $\eta\in\Gamma(\nu,\mu)$, we say that $(X,Y)$ is a coupling of $\nu$ and $\mu$. Implicitly, $\W_{\rho}$ means $\W_{1,\rho}$, and $\W_p$ means $\W_{p,\rho}$ with $\rho(x,y)=|x-y|$. In fact, unless otherwise specified, in $\R^d$, the name \emph{Wasserstein distances} usually refers to cases where $\rho$ is equivalent to the Euclidean metric.

Besides, another interesting case is the discrete metric $\rho(x,y) = 2\1_{x\neq y}$, in which case the associated $\W_1$ distance is the total variation norm
\[\W_{\rho}(\nu,\mu) \ = \ \|\nu - \mu\|_{TV} \ = \ \inf_{X\sim \nu,Y\sim\mu} 2 \mathbb P \po X\neq Y\pf\,.\]

Given a Markov operator $Q$ acting on $\mathcal P_p(\R^d)$, it is easily seen by conditioning on the initial condition that, for $C>0$, the proposition 
\[ \forall \nu,\mu\in\mathcal P_p(\R^d)\,, \qquad \W_{p,\rho}(\nu Q,\mu Q) \ \leqslant \ C \W_{p,\rho}(\nu,\mu)\]
is equivalent to 
\[ \forall x,y\in\R^d\,, \qquad \W_{p,\rho}(  Q(x,\cdot), Q(y,\cdot)) \ \leqslant \ C \rho(x,y)\,.\]
Besides, the Jensen's inequality implies that $\W_{p,\rho} \leqslant \W_{q,\rho}$ whenever $p\leqslant q$, which together with the previous remark implies that a contraction of the $\W_{q,\rho}$ distance for some $q$ induces a contraction of $\W_{p,\rho}$ for all $p \leqslant q$.

For a distance $\rho$ on $\R^d$ that is equivalent to  the Euclidean metric, from \cite[Corollary 5.22 and Theorem 6.18]{VillaniOldNew}, there always exists an optimal coupling for $\W_{p,\rho}$, i.e. for $\nu,\mu\in \mathcal P_{p}(\R^d)$ there exists a coupling $(X,Y)$ of $\mu$ on $\nu$ on some probability space $\Omega$ such that $\W_{p,\rho}^p(\nu,\mu) = \mathbb E \po \rho^p(X,Y)\pf$, and moreover $(\mathcal P_{p}(\R^d),\W_{p,\rho})$ is a Banach space. 

\section{Main results}\label{Sec:mainresults}

As already mentioned, we are principally interested in the smooth and convex case, corresponding to the following conditions.

\begin{assu*}[\textbf{$\na$Lip}]
There exists $L>0$ such that for all $x,y\in \R^d$,  
\[|\na U(x)-\na U(y)| \leqslant L |x-y|\,. \]
\end{assu*}

\begin{assu*}[\textbf{Conv}]
There exists $m>0$ such that for all $x,y\in \R^d$, 
\[(x-y)\cdot \po \na U(x) - \na U(y) \pf \geqslant m|x-y|^2 \,. \]
\end{assu*} 
Under Assumptions~\refsmooth\ and \refconvex, $U$ admits a unique global minimum $x_\star$ and for all $x\in\R^d$, 
\begin{equation}\label{eq:x_star}
|\na U(x)|\leqslant L|x-x_\star|\qquad\text{and}\qquad m |x-x_\star|^2 \leqslant 2U(x) \leqslant L|x-x_\star|^2\,.
\end{equation}
The two conditions don't have the same status.  The condition~\refsmooth\ will be enforced in all the article. It classically implies the stability of the OBABO scheme for $\delta$ small enough (this will be a consequence of our results). Approximating diffusions with non-Lipschitz coefficient leads to various difficulties and can lead to truncated algorithm, see e.g. \cite{Bourabee1} and references within. The condition~\refconvex\ gives a simple condition to obtain nice explicit convergence rates, and will not be used in most of the bias error analysis.

\subsection{Dimension free convergence rates}\label{Sec:mainresOBABO}
Recall $P$ is the Markov operator corresponding to the OBABO chain introduced in Section~\ref{Sec:OBABO}. In this section, we state our results obtained on the long-time behavior and regularization of this chain by itself, i.e. without referring to the continuous-time limit.

Under the the convex/smooth assumption, and for a suitable choice of the damping parameter, the Langevin diffusion is contractive, in the sense that its deterministic drift contracts the distances \cite{Malrieu,Chatterji1}. Hence, starting from two different initial states, the parallel coupling (i.e. considering the same Brownian noise) of two processes yields a deterministic contraction. The problem is quite more involved in a non-convex case where one has to take advantage of the noise to overcome potential barriers, in which case a combination of parallel and mirror (using reflected Brownian noises)  couplings and a suitable concave transformation of the distance can give a contraction in some modified $\mathcal W_1$ distance, see \cite{EberleGuillinZimmer,Chatterji2,dalalyan2}. An alternative argument in the non-convex case to get a $\mathcal W_2$ hypocoercive contraction is to combine an entropic hypocoercive contraction with a Wasserstein/entropy regularization, as in \cite[Theorem 2]{MonmarcheGuillin} (this gives sharper results than coupling arguments in the low temperature regime, see \cite{MonmarcheRecuitHypo}).

In the present convex case, the OBABO chain being an approximation of the Langevin diffusion, a simple parallel coupling does the job, and we get the following. 

\begin{thm}\label{thm_main:Wpcontract}
Under Assumptions~\refsmooth\ and \refconvex, suppose moreover that \nv{$\gamma \geqslant 2\sqrt{L}$ and that $\delta \leqslant m/(33\gamma^3)$}. Then, for all $p\geqslant 1$, $n\in \N$ and all $\nu,\mu\in\mathcal P_p(\R^{2d})$, 
 \[\W_p\po \nu P^n,\mu P^n\pf \ \leqslant \  K_1 \po 1-\delta \kappa\pf^{n/2} \W_p(\nu,\mu)\]
 with
 \[K_1\ \ = \ \nv{\sqrt{3}\max\po L^{1/2},L^{-1/2}\pf} \,,\qquad \kappa \ = \ \frac{m}{3\gamma} \,.\]
 Moreover, $P$ admits a unique invariant probability distribution $\pi_\delta$, and $\pi_\delta\in  \mathcal P_p(\R^{2d})$ for all $p\geqslant 1$. 
\end{thm}

This is proven in Section~\ref{sec:proofContract}, see Corollary~\ref{Cor:Wpcontract}.

\begin{rmq}
The restriction on $\gamma$ is consistent with all the works that studied the contraction of the parallel coupling for the continuous-time Langevin diffusion \cite{Malrieu,Chatterji1,dalalyan1,Zajic} (see also \cite{DoucetHMC} for the Randomized  HMC algorithm, since this is exactly the same computations after taking expectations). 
 \nv{In fact, as proven in  \cite[Proposition 4]{MonmarcheContraction}, for the continuous-time Langevin process, a Wasserstein contraction as in the proof of Theorem~\ref{thm_main:Wpcontract} holds for all potentials $U$ satisfying \refsmooth\ and \refconvex if and only if $L-m < \gamma (\sqrt{L}+\sqrt{m})$. When $m\ll L$ our condition $\gamma\geqslant 2\sqrt{L}$ is thus similar to  the optimal one up to a factor 2. In fact, it is clear that combining the proof of \cite[Proposition 4]{MonmarcheContraction} and of Theorem~\ref{thm_main:Wpcontract}, a Wasserstein contraction holds for the OBABO chain under the same optimal condition on $\gamma$ as in the time-continuous case, but then $\kappa$ and the condition on $\delta$ are less nice and this doesn't improve the convergence rate  in terms of the  dependency  on $m$ and $L$.}

The choice $\gamma= 2\sqrt{L}$ maximizes our bound on the contraction rate. \nv{In \cite{dalalyan1}, the optimal rate for the continuous-time process is proven to be $m/\gamma$, obtained with the choice $\gamma=\sqrt{m+L}$, which means in  the regime $m\ll L$ we are simply missing a factor $6$ from the optimal bound. Again, this factor $6$ can of course be reduced to get arbitrarily close to the optimal bound simply by following the proof and keeping sharper constants, to the cost of a stronger condition on $\delta$, and without changing the order in terms of $m$ and $L$}. 

Besides, as $\gamma\rightarrow +\infty$, we get a convergence rate  of order $1/\gamma$, which is sharp and well-known in the case of the continuous-time Langevin diffusion.

\nv{It is clear that a condition on $\delta$ is needed, since the process is not even stable when $\delta$ is too large. Here the condition $\delta \leqslant m/(33\gamma^3)$ is not sharp, it is used to get a simple expression. Anyway, we have in mind the regime where $m,L,\gamma$ are independent from the dimension and, as we will see in the next section, $\delta$ should vanish as $d\rightarrow +\infty$, so that the condition on $\delta$ is always satisfied in high dimension, which is why we chose simplicity over sharpness to state this condition.
}
\end{rmq}

As an hypoelliptic (non-elliptic) diffusion, the continuous-time Langevin diffusion~\eqref{Eq:Continu-Langevin} enjoys many nice regularization properties. In particular, under Assumption~\refsmooth, \cite[Corollary 4.7(2) and Remark 4.1]{GuillinWang} implies an instantaneous $\mathcal W_2$/relative entropy regularization, which together with the Pinsker inequality implies that
\[\|\nu P_t - \mu P_t\|_{TV} \ \leqslant \ \frac{C}{\min(1,t^{3/2})} \W_2(\nu,\mu)\]
for some $C>0$ for  all $\nu,\mu\in\mathcal P_2(\R^{2d})$ and $t>0$, where $(P_t)_{t\geqslant 0}$ is the transition semigroup associated to \eqref{Eq:Continu-Langevin}. For elliptic diffusions the result would typically hold with $\sqrt t$, but for the Langevin diffusion, there is no direct noise in the position and thus $t^{N+1/2}$ is to be expected where $N$ is the number of times one has to consider Lie brackets to fulfill Hörmander's condition (here $N=1$).

We now state a similar result, but at the level of the discrete time chain $P$. Convexity is not assumed.  We are not aware of any similar result for a Markov chain based on the discretization of an hypoelliptic non-elliptic diffusion (although, as mentioned in the introduction, it holds indeed in many cases).

\begin{prop}\label{prop_main:W/TV}
Under Assumption~\refsmooth, for all $\nu,\mu \in \mathcal P_1(\R^{2d})$,
\[\|\nu P - \mu P \|_{TV} \ \leqslant \ \frac {K_2}{\delta^{3/2} }  \W_1(\nu,\mu)\,,\]
with
\[K_2 \ = \ \sqrt{\frac{ \delta }{2\pi \po 1 -\eta^2\pf }}\max\po 2+ 3\delta^2 L/2, 3\delta\pf \nv{\  \underset{\delta\rightarrow 0}{\longrightarrow}\ \sqrt{\frac{2}{\pi\gamma} }   } \,.\]
\end{prop}

This is proven in Section~\ref{Sec:proofWTV}, see Proposition~\ref{prop:W/TV}.

\begin{rmq}
We retrieve for $\delta\rightarrow 0$ the expected $\delta^{3/2}$.  In fact, more precisely, as can be seen in Proposition~\ref{prop:W/TV} below, we get the scaling $\delta^{3/2}$ for the positions and $\sqrt\delta$ for the velocities.
\end{rmq}

Combining the two previous results naturally yields a convergence rate for the total variation distance. Following \cite{RobertTweedie,QinHobert}, we consider the total variation convergence rate of $P$ in the sense of: 
\[r_*(P) \ :=  \ \exp \po \sup_{z\in\R^{2d}} \limsup_{n\rightarrow +\infty} \frac{\ln \|\delta_z P^n - \pi_\delta\|_{TV}}{n}\pf\,.\]
An immediate consequence of Theorem~\ref{thm_main:Wpcontract} and of Proposition~\ref{prop_main:W/TV} is the following.
\begin{cor}\label{Cor:r*P}
Under the conditions of Theorem~\ref{thm_main:Wpcontract}, for all $\nu,\mu \in \mathcal P_1(\R^{2d})$ and all $n\geqslant 1$,
\[\|\nu P^n - \mu P^n \|_{TV} \ \leqslant \ \frac{K_1K_2}{\delta^{3/2}}\po 1-\delta \kappa\pf^{(n-1)/2}  \W_1(\nu,\mu)\,,\]
with $K_1,\kappa$ and $K_2$ given in Theorem~\ref{thm_main:Wpcontract} and Proposition~\ref{prop_main:W/TV}. In particular, $r_*(P) \leqslant 1 - \delta \kappa/2$. 
\end{cor}

The most classical settings to obtain bounds on the total variation convergence rate of a Markov chain is the combination of a Lyapunov (drift) and a local Doeblin (minoration) conditions. Nevertheless, the recent \cite{QinHobert} rigorously establishes that, in a variety of cases, no drift/minoration condition can yield a reasonable bound on $r_*(P)$ (here only one-step minoration conditions are concerned, of the form $\inf_{z\in\mathcal C}P(z,\cdot) \geqslant \varepsilon\nu$ for some probability $\nu$ and some $\varepsilon>0$).  More precisely, it is clear that the results of \cite{QinHobert} apply to the OBABO sampler (essentially adapting the proof of \cite[Proposition 15]{QinHobert} for the MALA sampler), which proves that it is impossible to prove simply with a drift/minoration condition that, if $\delta$ scales polynomially with the dimension, then $r_*(P)$ is bounded away from $1$ polynomially in   the dimension. In Corollary~\ref{Cor:r*P}, we have used the combination of a Wasserstein convergence and a Wasserstein/total variation regularization, called in \cite{QinHobert,RobertRosenthal,MadrasSezer} a one-shot coupling, which is indeed presented in \cite{QinHobert} as an alternative to drift/minoration arguments.

\bigskip

We finish this section with a concentration result implied by the Wasserstein contraction of Theorem~\ref{thm_main:Wpcontract}, \nv{or more precisely by the deterministic contraction at the core of the proof of this result (see Proposition~\ref{prop:contraction} below)}. We say that a probability distribution $\nu$ on $\R^{d}$ satisfies a logarithmic Sobolev (or simply log-Sobolev) inequality with constant $C$ if, for all bounded Lipschitz $\varphi$, 
\[\int_{\R^d} \varphi^2 \ln \po \varphi^2\pf \dd \nu - \po  \int_{\R^d} \varphi^2 \dd \nu \pf \ln\po  \int_{\R^d} \varphi^2 \dd \nu \pf   \ \leqslant \ C \int_{R^d} |\na \varphi|^2 \dd \nu\,. \]
Such inequality yields many useful information, in particular related to concentration of measure, see e.g. \cite{BakryGentilLedoux,LedouxConcentration}. As a consequence of the Brascamp-Lieb inequality or of the Bakry-Emery curvature criterion \cite{BakryGentilLedoux}, it is well known that, under the condition~\refconvex, the Gibbs measure $\pi$ satisfies such an inequality with constant $2\max(1,1/m)$. However, the invariant measure of $P$ is not explicit. Nevertheless, the positive Wasserstein curvature  (in the sense of \cite{Ollivier,Joulin}) obtained in Theorem~\ref{thm_main:Wpcontract} (or more precisely in Proposition~\ref{prop:contraction}) yields a discrete-time Bakry-Emery condition, which in turn yields the following.

\begin{thm}\label{thm_main:concentration}
Under the conditions of Theorem~\ref{thm_main:Wpcontract}, $\pi_\delta$ satisfies a log-Sobolev inequality with constant 
\[C_{lS} \ = \    \frac{6}{m} K_1^2\gamma^2 \po 1 + \delta + \delta^2   L/2\pf^2 \,.\]
Moreover, if $\mu_0 \in\mathcal P(\R^{2d})$ satisfies a log-Sobolev inequality with constant $C'>0$ and if $(Z_k)_{k\in \N}$ is an OBABO  chain with $Z_0\sim \mu_0$, then for all $1$-Lipshitz functions $\varphi$ on $\R^{2d}$, $n\geqslant 1$ and $u\geqslant 0$,
\[\mathbb{P}\po \left|\frac1n\sum_{k=1}^n   \phi(Z_k) - \pi_\delta(\phi)\right| \geqslant u +   \frac{2K_1}{n\kappa\delta}   \W_1(\mu_0, \pi_\delta) \pf \\ 
 \  \leqslant \  2  \exp \po - \   \frac{  n u^2 \kappa \delta }{4  C_{lS}+ 4K_1^2C'/(n\kappa\delta)}\pf\,.\]
\end{thm}
\begin{rmq}\label{rmq:initialcondition}
If $\mu_0$ is a Dirac measure, it satisfies a log-Sobolev inequality with constant $C'=0$. In any cases, we can bound $\W_1(\mu_0, \pi_\delta)$ by $\W_1(\mu_0, \pi)+ \W_1(\pi, \pi_\delta)$, and then it is possible to chose  $\mu_0 = \delta_{(\tilde x_\star,0)}$ or $\delta_{\tilde x_\star}\otimes \mathcal N(0,I_d)$ (which satisfies a log-Sobolev inequality with constant $2$), where $\tilde x_\star$ is an approximation of $x_\star$ obtained through  a deterministic optimization scheme. As can be seen for instance in Lemma~\ref{lem:moments}, this  yields an initial distance $\W_1(\pi,\mu_0)$ of order $\sqrt d + |x_\star-\tilde x_\star|$. The remaining term $\W_1(\pi, \pi_\delta)$ (that also controls $|\pi_\delta(\varphi)-\pi(\varphi)|$) is studied in the next section.
\end{rmq}

\begin{rmq}
In fact, a $\W_1$ contraction (which does not necessarily implies a log-Sobolev inequality) is sufficient to get a similar result, see \cite[Corollary 2.6]{Djellout} and \cite{Joulin}.

\end{rmq}

\subsection{Equilibrium bias and efficiency}\label{Sec:mainresBias}

The invariant measure $\pi_\delta$ of $P$ is not explicit. Since the OBABO sampler converges to the Langevin diffusion as $\delta$ vanishes, we expect $\pi_\delta$ to converge to $\pi$. There are many classical ways to quantify this; in view of the previous long-time convergence results, in this work we are naturally interested by an estimation of the Wasserstein and total variation distances  between $\pi_\delta$ and $\pi$. 
 
 Recall the definition of the Verlet map $\Phi_V$ in Section~\ref{Sec:MOBABO}, and consider $P_V$ the corresponding transition operator, i.e. $P_V\varphi(x,v) = \varphi(\Phi_V(x,v))$ (equivalently, $\nu P_V = \Phi_V \sharp \nu$). From the contraction of Wasserstein distances and the Wasserstein/total variation distance regularization property, it is not difficult to get the following.

\begin{prop}\label{prop_main:nupi1}
Under the conditions of Theorem~\ref{thm_main:Wpcontract}, for all $p\geqslant 1$,
\[\mathcal W_{p}\po \pi_\delta ,\pi\pf \ \leqslant \ \frac{2 K_1}{\delta \kappa} \mathcal W_{p}\po \pi   ,\pi P_V\pf\]
and for all $n\in\N_*$, 
 \begin{eqnarray*}
 \|\pi_\delta - \pi\|_{TV} 
 & \leqslant &  \frac{ K_1 K_2}{\delta^{3/2}} (1-\delta\kappa)^{(n-1)/2} \W_1(\pi_\delta , \pi )+ n\|\pi  P_V - \pi  \|_{TV} 
 \end{eqnarray*}
 where $K_2$ is given in Proposition~\ref{prop_main:W/TV}.
\end{prop}

This is proven in Section~\ref{Sec:CoupleVerlet}.

  \begin{rmq}
  Using that $(x',v')=\Phi_V(x,v)$ is a unitary change of coordinates (see \cite[Problem 1.5]{Tuckerman}) and that it is reversible up to a reflection of the velocity (which preserves the Hamiltonian $H$) we see that
  \[\int_{\R^{2d}} \varphi\po \Phi_V(x,v)\pf e^{-H(x,v)}\dd x \dd v \ = \ \int_{\R^{2d}} \varphi(x,v) e^{-H\po \Phi_V(x,-v)\pf}\dd x \dd v\,.\]
  In other words, $\pi P_V$ is the probability law with density proportional to $\exp(-H\circ \Phi_V \circ \Phi_R)$.
  \end{rmq}

This leads us to the study of the Verlet step. In order to control the distance from $\pi P_V$ to $\pi$, a natural way is to consider a Markov operator $Q$ that leaves $\pi$ invariant and to control the distance between $\pi P_V$ and $\pi Q = \pi$ by constructing a coupling of a transition of $P_V$ with a transition of $Q$, starting from the same initial condition distributed according to $\pi$.  If $Q$ is close to $P_V$ then we will be able to do so while keeping the two chains close.  Since $P_V$ has been obtained as the discretization of the Hamiltonian dynamics, a   natural candidate for $Q$  is $Q_\delta$ where $(Q_t)_{t\geqslant 0}$ is the transition semi-group of this deterministic flow, which indeed leaves $\pi$ invariant. Comparing the discrete-time Markov chain with its continuous-time limit is indeed the main argument in similar works such as \cite{Chatterji1,Chatterji2,durmus2019,dalalyan1}. However, we remark that another candidate is given by the Metropolis-adjusted Verlet transition  $P_{MV} = P_R P_{MH}$ introduced in Section~\ref{Sec:MOBABO}.  Indeed, $\pi$ is invariant for $P_{MV}$, and $P_V$ and $P_{MV}$ only differ by an accept/reject step. Comparing $P_V$ with $Q_\delta$ or with $P_{MV}$ leads to two different kind of results. On the one hand, $P_V$ and $Q_\delta$ are deterministic, so there is no real coupling here, only a deterministic numerical error, in particular it behaves well with $\W_p$ for all $p\geqslant 1$. On the other hand, it is not suitable to bound the total variation distance since the probably that the deterministic flow and the numerical integrator are equal after one step is in general zero. On the contrary, by design, a transition of $P_V$ and $P_{MV}$ gives exactly the same result provided the step is accepted in $P_{MV}$, which enables to get information on the total variation distance. Nevertheless, in case of rejection, because of the velocity reflection, although the initial condition is the same for the two chains, the distance instantaneously get large. This makes the comparison with $P_{MV}$ less and less efficient for bounding $\W_p$ as $p$ increases. In a word, when comparing $P_V$ with $Q_\delta$, we get two points that are very close (but distinct) with probability $1$ while, when comparing $P_V$ with $P_{MV}$, we get two points that are equal with large probability but distant otherwise.

\medskip

This analysis is conducted in Section~\ref{Sec:CoupleVerlet}. We now present its results.

\medskip

A first remark is that the main difference between the OBABO sampler and the Markov chain studied in \cite{Chatterji1} is that the former  is based on a second-order integrator, at least formally. Nevertheless, a Taylor expansion of the Verlet algorithm to an order larger than one requires informations on derivatives of $U$ higher than two. As a consequence, relying only on Assumption~\refsmooth, we can only get results similar to \cite{Chatterji1} for the Wasserstein distances. 

\begin{prop}\label{prop_main:erreur1}
 Under the conditions of Theorem~\ref{thm_main:Wpcontract}, for all $p\geqslant 1$,
\begin{eqnarray*}
\W_p\po \pi,\pi_\delta\pf  & \leqslant & \delta  \sqrt{d+(p-2)_+} \frac{2 K_1K_3}{ \kappa} \\
\|\pi - \pi_\delta\|_{TV} & \leqslant & \delta d \err{\po 2+   |\ln \po  \delta^{3} d\pf|   \pf } \po  \frac{2 K_1^2 K_2K_3}{ \kappa\sqrt{1-\delta \kappa}} + \frac{K_4}{\kappa}  \pf 
\end{eqnarray*}
where $K_2,K_3,K_4$ are given in Propositions~\ref{prop_main:W/TV} and \ref{prop:WppiPV} \nv{(in particular $K_3$ and $K_4$ converge to $L$ as $\delta$ vanishes)}.
\end{prop}

Proposition~\ref{prop_main:erreur1} is a corollary of Propositions~\ref{prop_main:nupi1} and \ref{prop:WppiPV}, as proven at the end of Section~\ref{Sec:CoupleVerlet}.

\medskip

In order to see the benefits of the second-order approximation, some conditions on higher derivatives of $U$ have to be enforced. Ideally, we would like conditions that are general, simple, and under which the OBABO sampler proves to be efficient. Of course these three criteria are contradictory and, as a compromise, we study several conditions to get a partial picture of the behavior of the algorithm. We could state a technical ad hoc condition containing exactly what is used in the proof, but then such a condition would have to be checked on particular cases. We will rather focus on two simple and general conditions, which provides a simple proof that there exist situations for which the bias of the chain is of second order in $\delta$ with possibly a nice dependence in the dimension.

 The simplest condition one can think of, considered in other works like \cite{durmus2019,dalalyan1,Chatterji3}, is that the Hessian of $U$ is Lipschitz:
\[\exists L_2>0\ \text{ such that }\ \forall x,y\in\R^d\,,\qquad  |\na^2 U(y) -\na^2 U(x)| \ \leqslant \ L_2 |x-y|\,.\]
 \nv{In fact, we will work under a possibly weaker generalization of this condition:} 

\begin{assu*}[$\mathbf{\na^2 pol(}\ell\mathbf{)}$]
There exist $\ell\geqslant 2$, $L_\ell>0$ and $x_\star \in\R^d$ such that, for all $x,y\in\R^d$, 
\begin{equation}\label{eq:ordren}
|\na^2 U(y) - \na^2 U(x) | \  \leqslant \ L_\ell |x-y| \po 2+ |x-x_\star|^{\ell-2}+|y-x_\star|^{\ell-2}\pf/4\,.
\end{equation}
\end{assu*}
Note that, up to a change of the constant $L_\ell$, if Assumption~\refordren\ holds for some $x_\star$ then it holds with any other value of $x_\star$. As a consequence, implicitly, whenever Assumptions~\refconvex\ and \refordren\ are simultaneously satisfied, then we chose $x_\star$ in \refordren\ to be the global minimum of $U$ given by \refconvex.

As can be seen through a Taylor expansion, \refordren\ holds in particular if $\|\na^{(k)} U\|_\infty <+\infty$ for some $k\geqslant 3$. 

As we will see, we still get a second-order error in $\delta$ under Assumption~\refordren\ for any $\ell\geqslant 2$. Nevertheless, for $\ell>2$, due to the presence of additional moments, the dependence in the dimension worsen. However, as illustrated by the next condition, there are cases where, even though the local Lipschitz constant of $\na^2 U$ is polynomial, the OBABO sampler is super efficient.

\begin{assu*}[$\mathbf{\indep}$]
There exist $\ell\in \N$, $L_\ell>0$, $x_\star = (x_{\star,1},\dots,x_{\star,d})\in\R^d$, an orthogonal matrix $\mathcal Q\in\mathcal O(d)$ and, for $i\in\cco 1,d\ccf$, potentials $U_i\in\mathcal C^2 (\R)$  such that $U(\mathcal Q x) = \sum_{i=1}^d U_i(x_i)$ and for all $ i\in\cco 1,d\ccf$ and $x,y\in\R$,
\[ |U_i''(y)-U_i''(x)| \ \leqslant \ L_\ell |x-y|\po 2+ |x-x_{\star,i}|^{\ell-2}+|y-x_{\star,i}|^{\ell-2}\pf/4\,.\]
\end{assu*}

Similarly to Assumption~\refordren, when Assumptions~\refind\ and \refconvex\ are both satisfied, $x_\star$ is implicitly the minimum of $U$. Assumption~\refind\ means that the target distribution is separable, namely, in a suitable orthonormal basis, the coordinates of a random variable distributed accorded to $\pi$ are independent. It is important to note that we do not need to know $\mathcal Q$. Indeed, contrary to e.g. the Zig-Zag sampler \cite{BierkensFearnheadRoberts} or more generally Gibbs algorithms,  in the OBABO sampler, the basis does not play a particular role, i.e. if $(x_n,v_n)_{n\in \N}$ is an OBABO chain for some potential $U$ then  $(\mathcal Q^{-1}x_n,\mathcal Q^{-1}v_n)_{n\in \N}$ is an OBABO chain with potential $x\mapsto U(\mathcal Q x)$. As a particular case, Assumption~\refind\ is satisfied for all Gaussian distributions (with $\ell=2$ and $L_\ell=0$).

Although the independence assumption is very restrictive, it is the classical condition under which scaling limits are established for MCMC algorithms, see e.g. \cite{DoucetHMC} and references within. Non-asymptotic bounds for the HMC algorithm for a separable target have been established in \cite{Mangoubi}, and under a similar but weaker condition in \cite{Mangoubi2}.

\begin{prop}\label{prop_main:biais}
Under the conditions of Theorem~\ref{thm_main:Wpcontract}, and considering $K_5,K_6,K_7,K_8$ as given by Proposition~\ref{prop:WppiPV}, \nv{(in particular
\[K_5 \rightarrow  \frac{L}{4\sqrt m}  
   +   \frac{L }{6} \,,\quad K_6 \rightarrow  2^{\ell-2}  \po 1 + \frac{1}{m^{\ell/2}}\pf\,,
   \quad K_7 \rightarrow  L\,,
   \quad K_8 \rightarrow  2^{\ell+1}  \po 1 + \frac{1}{m^{(\ell+1)/2}}\pf \]
as $\delta \rightarrow 0$), the following holds:
}

\begin{enumerate}

\item If Assumption~\refordren\ holds then, for all $p\geqslant 1$,
 \begin{eqnarray*}
 \mathcal W_{p}\po \pi_\delta ,\pi\pf & \leqslant & \delta^2\po d + \ell p-2\pf^{\ell/2}\frac{2 K_1 \po K_5   + L_\ell K_6 \pf}{ \kappa}   
 \\
 \|\pi_\delta - \pi\|_{TV} 
 & \leqslant & \delta^2 \err{\po 2+ |\ln\po  \delta^{3}(d+\ell-1)\pf| \pf} \po d + \ell -1\pf^{(\ell+1)/2}  \\
 & & \ \times \    \frac{2 K_1^2K_2 \po K_5   + L_\ell K_6 \pf + K_7   + L_\ell K_8  }{ \kappa} \,.
\end{eqnarray*}

\item If Assumption~\refind\ holds then, for all $p\geqslant 2$,
 \begin{eqnarray*}
 \mathcal W_{p}\po \pi_\delta ,\pi\pf & \leqslant & \delta^2 \sqrt{d}\frac{2 K_1 \po K_5   + L_\ell K_6 \pf \po  \ell p-1\pf^{\ell/2}}{ \kappa}  
 \\
 \|\pi_\delta - \pi\|_{TV}  & \leqslant &  \delta^2 d \err{\po 2+ | \ln\po    \delta^{3}\ell\pf| \pf} \ell  ^{(\ell+1)/2}     \frac{2 K_1^2K_2 \po K_5   + L_\ell K_6 \pf + K_7   + L_\ell K_8  }{ \kappa} \,.
\end{eqnarray*}
\end{enumerate}
\end{prop}

This is proven at the end of Section~\ref{Sec:CoupleVerlet}.

\nv{
\begin{rmq}
In the Gaussian case, $\pi_\delta$ is an explicit Gaussian measure and thus $\mathcal W_2(\pi,\pi_\delta)$ is known. More precisely, in that case, we get $\mathcal W_2(\pi,\pi_\delta) \simeq \delta^2 \sqrt{dL}/4$ as $\delta \rightarrow 0$ (we refer to the upcoming work \cite{MoiRomberg} for the detailed analysis of the Gaussian case), which means the $\delta^2 \sqrt d$ dependency under Assumption~\refind\ is sharp.

On the contrary, for the dependency in $d$ of the total variation  bound, the results of \cite{Chatterji3} give a bound of order $\sqrt d$ under the condition \refordren\ with $\ell=2$, while we only get $d^{3/4}$. We discuss this in more details in Section~\ref{Sec:otherworks}.
\end{rmq}
}
\nv{
Although one can design academic examples where  \refordren\ holds for some $\ell >2$ but not for $\ell=2$, while \refsmooth\ holds, this does not correspond to practical cases of interest. The reason why we decided to work under a framework with possibly $\ell>2$ is the following. For potentials that behave like $U(x) = |x|^\alpha$ with $\alpha>2$ at infinity, \refordren\ holds with $\ell>2$ but not $\ell=2$. On the other hands \refsmooth\ doesn't hold, but it can be replaced by a condition similar to \refordren\ but on $\na U$ rather than $\na^2 U$. Moreover the OBABO integrator becomes unstable and should be replaced by an adaptive time-step/tamed version. Adapting our analysis to this case, similar but more involved, is out of the scope of the present work, but  we wanted to highlight that working with such local Lipschitz-continuity conditions is possible and possibly  leads to a higher dependency in the dimension of the bias estimates due to additional moments, but not necessarily, as illustrated by the separable case.
}

 \subsection{The Metropolis-adjusted chain}\label{Sec:mainresMOBABO}

Since the continuous-time Langevin diffusion has, in the smooth and convex case, a convergence rate that is independent from the dimension, and as seen in Theorem~\ref{thm_main:Wpcontract},  the OBABO chain has a convergence rate that depends on the dimension only, possibly, through  the time-step $\delta$. Hence, the number of steps of the chain required in order to sample the target measure is of order $1/\delta$ (up to some logarithmic terms). On the other hand,  the rejection probability in the OM(BAB)O scheme is of order $\delta^2$. This means that,  as $\delta$ vanishes, there is a high probability that no rejection occurs during the whole simulation. Consequently, results obtained for the unadjusted OBABO chain are straightforwardly transfered to the OM(BAB)O chain. For simplicity we focus on the total variation distance and confidence intervals, but similar adaptations could be conducted for Wasserstein distances.

In the following, under assumption~\refconvex, we denote $z_\star=(x_\star,0)$ (which is  the minimum of $H(x,v)=U(x)+|v|^2/2$) and, for $p\geqslant 1$ and $\nu \in\mathcal P_p(\R^{2d})$, consider the moments
\[M_{\nu,p} \ = \ \mathbb E_{\nu}\po |Z-z_\star|^p\pf\,.\]

\begin{prop}\label{prop_main:Metropolis}
Under the conditions of Theorem~\ref{thm_main:Wpcontract}, for all   $n\in\N_*$ and $\nu,\mu\in\mathcal P_{2}(\R^{2d})$, 
\[\| \nu P_M^n - \mu P_M^n \|_{TV}   \leqslant \| \nu P^n - \mu P^n \|_{TV} + \mathcal E(\nu,\delta,n,d)+ \mathcal E(\mu,\delta,n,d)  \,,\]
where $\mathcal E$ is such that: 
\begin{enumerate}
\item For all $\nu  \in\mathcal P_{2}(\R^{2d})$, considering $K_9$ as in Corollary~\ref{cor:nuPkPO},
\[\mathcal E(\nu,\delta,n,d)  \  \leqslant \  \delta K_9 \po  M_{\nu,2}+    \delta n d \pf  \,.\]
\item If  condition~\refordren\ holds and $\nu\in\mathcal P_{\ell+1}(\R^{2d})$, considering $K_{10}$ as in Corollary~\ref{cor:nuPkPO},
\[\mathcal E(\nu,\delta,n,d) \   \leqslant\   \delta^2 K_{10} \po  M_{\nu,\ell+1} +   \delta n \po d+ \ell-1\pf^{(\ell+1)/2} \pf  \,.\]
\item If  condition~\refind\ holds, and $\nu\in\mathcal P_{\ell+1}(\R^{2d})$, considering $K_{10}$ as in Corollary~\ref{cor:nuPkPO} and, for $i\in\cco 1,d\ccf$, denoting by $\nu_i$ the law of $Y_i$ where $Y =\mathcal Q^{-1} X$ with $X\sim \nu$, 
\[\mathcal E(\nu,\delta,n,d) \   \leqslant\  \delta^2 K_{10} \po \sum_{i=1}^d   M_{\nu_i,\ell+1} +    \delta n d \ell^{(\ell+1)/2} \pf  \,.\]
 \end{enumerate}
\end{prop}

This is proven at the end of  Section~\ref{Sec:proofMOBABO}. \nv{To fix ideas, taking first the limit as $\delta\rightarrow 0$ and then the leading terms as $L,L_\ell\rightarrow +\infty$ and $m\rightarrow 0$ (with $L_{\ell}\ll L^2$), we get 
\[K_9 \ \simeq \ \frac{18L^2 \gamma}{m^2} \,, \qquad 
K_{10} \ \simeq \ \frac{ 9 \times 12^{(\ell-1)/2} L^{(\ell+5)/2} \gamma}{m^{(\ell+3)/2}}            \,.\]
}

\medskip

Applied with $\mu = \pi$ (so that $\mu P_M^n = \pi$) and in combination with Corollary~\ref{Cor:r*P}, this gives a non-asymptotic bound on $\|\nu P_M^n-\pi\|_{TV}$.

\begin{rmq}\label{rmq:localCouplMetropolis}
Contrary to Corollary~\ref{Cor:r*P} for the OBABO chain, Proposition~\ref{prop_main:Metropolis} does not provide a bound on the total variation convergence rate of the OM(BAB)O chain. Nevertheless,  combining Proposition~\ref{prop_main:Metropolis} with Corollary~\ref{Cor:r*P} yields local coupling estimates, i.e. for all $R>0$, a suitable choice of $\delta$ and $n$ ensures that
\[\forall z,z'\in\{y\in\R^{2d},\ |y-z_\star|\leqslant R\}\,,\qquad \| P_M^n(z,\cdot) - P_M^n(z',\cdot) \|_{TV}   \leqslant \frac12\,.\]
Contrary to the one-step Doeblin condition considered in \cite{QinHobert}, these  $n$-steps coupling bounds have a reasonably nice scaling with $\delta$ and $d$.
\end{rmq}

Similarly, the non-asymptotic confidence intervals of Theorem~\ref{thm_main:concentration} can be transfered to the Metropolis-adjusted chain:

\begin{thm}\label{thm_main:concentration_Metropolis}
Under the conditions of Theorem~\ref{thm_main:Wpcontract}, if $\mu_0 \in\mathcal P(\R^{2d})$ satisfies a log-Sobolev inequality with constant $C'>0$ and if $(Z_k)_{k\in \N}$ is an OM(BAB)O  chain with $Z_0\sim \mu_0$, then for all $1$-Lipschitz functions $\varphi$ on $\R^{2d}$, $n\geqslant 1$ and $u\geqslant 0$,
\begin{multline*}
\mathbb{P}\po \left|\frac1n\sum_{k=1}^n   \phi(Z_k) - \pi(\phi)\right| \geqslant u + \frac{2K_1}{n\kappa\delta}   \W_1(\mu_0, \pi_\delta)  + \W_1(\pi,\pi_\delta)\pf \\ 
 \  \leqslant \  2  \exp \po - \   \frac{  n u^2 \kappa \delta }{4  C_{lS}+ 4K_1^2C'/(n\kappa\delta)}\pf + \mathcal E(\mu_0,\delta,n,d) \,.
 \end{multline*}
where $C_{lS}$ is given in Theorem~\ref{thm_main:concentration} and $\mathcal E(\nu,\delta,n,d) $ is  as in Proposition~\ref{prop_main:Metropolis}.
\end{thm}

This is proven at  then end of Section~\ref{Sec:proofMOBABO}.

\subsection{Conclusion and related works}\label{Sec:otherworks}

Let us summarize the previous results, focusing on the behavior as $\delta \rightarrow 0$ and $d\rightarrow +\infty$ for fixed values of $p,m,L,L_\ell,\ell$ and $\gamma = \nv{2\sqrt L}$. We suppose that the initial condition $\mu_0$ is such that, for $p\geqslant 1$, $ \mathcal W_p(\mu_0,\pi) + M_{\mu_0,p}^{1/p} \leqslant \tilde K_p   \sqrt{d}$ for some $\tilde K_p$ that does not depend on $d$, which is feasible in practice (see Remark~\ref{rmq:initialcondition} and Lemma~\ref{lem:moments}). \nv{For the dependency on $L,m,L_\ell$, we only write the leading terms when $L,L_\ell\rightarrow +\infty$ and $m\rightarrow 0$.}

For a desired accuracy $\varepsilon>0$, we denote
\begin{eqnarray*}
n_{\varepsilon,p} & = & \inf\{n\in\N,\ \W_p(\mu_0 P^n,\pi) \leqslant \varepsilon\}\\
 n_{\varepsilon,TV} & = & \inf\{n\in\N,\ \|\mu_0 P^n-\pi\|_{TV} \leqslant \varepsilon\}\\
  n_{\varepsilon,M} & = & \inf\{n\in\N,\ \|\mu_0 P_M^n-\pi\|_{TV} \leqslant \varepsilon\}\,.
\end{eqnarray*}
Choosing those as non-asymptotic criteria for efficiency is debatable, see the discussion in \cite{dalalyan2} in particular for a different scaling for the Wasserstein distances, nevertheless we stick to these definitions for comparison with previous works.  Table~\ref{TableDiscret} gathers the results on the discretization bias and Table~\ref{TableEfficace} the efficiency bounds that are obtained from our various results. \nv{For simplicity, in Table~\ref{TableEfficace}, concerning the dependency on $L,m$ and $L_\ell$, we only consider the case where the term with $L_\ell$ is negligible with respect to the other one (as in the Gaussian case where $L_\ell=0$) and we don't write the logarithmic terms in these variables.}

\begin{table}[h!]
\[\begin{array}{|c|c|c|c|}
\hline 
 & \text{\refsmooth, \refconvex}  & \text{\refsmooth, \refconvex, \refordren}  &  \text{\refsmooth, \refconvex, \refind}   
 \\ \hline
 & & & \\ 
 \mathcal W_p(\pi,\pi_\delta) & \delta \sqrt{d} \nv{\frac{L^{2}}{m}} & \delta^2 d^{\ell /2} \nv{\po \frac{L^2}{m^{3/2}} + \frac{\sqrt{L} L_\ell}{m^{1+\ell/2}}\pf} & \delta^2 \sqrt{d} \nv{\po \frac{L^2}{m^{3/2}} + \frac{\sqrt{L} L_\ell}{m^{1+\ell/2}}\pf}  \\
& & &  \\ \hline
 & & & \\
 \|\pi-\pi_\delta\|_{TV} & \delta d  \nv{\err{|\ln \po\delta^3 d \pf |}} \nv{\frac{L^{9/4}}{m} } & \delta^2 d^{(\ell+1)/2}\nv{\err{|\ln \po \delta^3 d \pf| } \po \frac{L^{9/4}}{m^{3/2}} + \frac{L^{5/4}L_\ell}{m^{(\ell+3)/2}} \pf }   &  \delta^2 d \err{|\ln \delta|} \nv{ \po \frac{L^{9/4}}{m^{3/2}} + \frac{L^{5/4}L_\ell}{m^{(\ell+3)/2}} \pf } \\
& & &  \\ \hline
\end{array}\]
\caption{Leading terms of the bounds on the equilibrium bias respectively obtained in Propositions~\ref{prop_main:erreur1} and \ref{prop_main:biais}.}\label{TableDiscret}
\end{table}

\begin{table}[h!]
\[\begin{array}{|c|c|c|c|}
\hline 
 & \text{\refsmooth, \refconvex}  & \text{\refsmooth, \refconvex, \refordren}  &  \text{\refsmooth, \refconvex, \refind}   
 \\ \hline
 & & & \\
 n_{\varepsilon,p} & \displaystyle{\frac{\sqrt{d}}{\varepsilon}  \ln \po \frac{ d }{\varepsilon}\pf} \nv{\frac{L^{5/2}}{m^2}} & \displaystyle{\frac{d^{\ell /4}}{\sqrt{\varepsilon}}  \ln \po \frac{ d }{\varepsilon}\pf}  \nv{\frac{L^{3/2}}{m^{7/4}}}  & \displaystyle{\frac{d^{1 /4}}{\sqrt{\varepsilon}} \ln \po \frac{ d }{\varepsilon}\pf}\nv{\frac{L^{3/2}}{m^{7/4}}} \\
& & &  \\ \hline
 & & & \\
 n_{\varepsilon,TV} &  \displaystyle{\frac{d}{\varepsilon}  \ln^2 \po \frac{ d }{\varepsilon}\pf}\nv{\frac{L^{11/4}}{m^{2}}} & \displaystyle{\frac{d^{(\ell+1) /4}}{\sqrt{\varepsilon}}  \ln^2 \po \frac{ d }{\varepsilon}\pf} \nv{\frac{L^{13/8}}{m^{7/4}}} & \displaystyle{\frac{\sqrt d }{\sqrt{\varepsilon}}  \ln^2 \po \frac{ d }{\varepsilon}\pf} \nv{\frac{L^{13/8}}{m^{7/4}}} \\
 & & &  \\ \hline
  & & & \\
\nv{ n_{\varepsilon,M}} &  \nv{\displaystyle{\frac{d}{\varepsilon}  \ln^4 \po \frac{ d }{\varepsilon}\pf}\frac{L^{7/2}}{m^{4}}}  & \nv{\displaystyle{\frac{d^{(\ell+1) /4}}{\sqrt{\varepsilon}}  \ln^3 \po \frac{ d }{\varepsilon}\pf} \frac{L^{(\ell+9)/4}}{m^{(\ell+9)/4}}}   & \nv{\displaystyle{\frac{\sqrt d}{\sqrt{\varepsilon}}  \ln^3 \po \frac{ d }{\varepsilon}\pf} \frac{L^{(\ell+9)/4}}{m^{(\ell+9)/4}}}  \\
 & & &  \\ \hline
\end{array}\]
\caption{Leading terms of the efficiency bounds obtained with a suitable choice of $\delta$, based on    Propositions~\ref{prop_main:erreur1} and \ref{prop_main:biais} (for the equilibrium bias in the unadjusted case), Theorem~\ref{thm_main:Wpcontract} and Corollary~\ref{Cor:r*P} (for the convergence of the OBABO chain to equilibrium) and Proposition~\ref{prop_main:Metropolis} (for the OM(BAB)O chain).}\label{TableEfficace}
\end{table}

Let us now compare our results to  previous works.

\medskip

First, we mention that, starting back at least to the work of Talay \cite{Talay} (who was also already concerned with the discretized chain), there has been a substantial amount of works in the last two decades on obtaining quantitative convergence  rates for the \emph{continuous-time} (underdamped) Langevin diffusion \eqref{Eq:Continu-Langevin}, leading to the theory of hypocoercivity \cite{Villani2009,DMS2009,Herau2007}. In the convex and smooth case, the dimension-free convergence is established by the author in \cite{MonmarcheVFP} for mean-field particles using entropic hypocoercivity techniques (implying the Wasserstein convergence, see e.g. \cite{MonmarcheGuillin} for details), and later by \cite{Chatterji1,dalalyan1,Zajic} by direct coupling arguments in the context of MCMC. The Wasserstein contraction induced by the parallel coupling in the smooth and convex case was first established in \cite{Malrieu} (but without the explicit dependency on the dimension). In the non-convex case, quantitative results are obtained through an explicit combination of reflection and parallel coupling in \cite{EberleGuillinZimmer,dalalyan2,Chatterji2}. Using hypocoercive techniques, one of the first result on the underdamped Langevin algorithms motivated by stochastic algorithms (MCMC and simulated annealing) is established in  \cite{MonmarcheRecuitHypo}, focusing on low-temperature rather than high-dimensional estimates. In the recent \cite{Chatterji3}, a similar result is established, focusing on high-dimensional MCMC (in both \cite{Chatterji3} and \cite{MonmarcheRecuitHypo}, the convergence is obtained in term of the log-Sobolev constant of the target distribution; the main difference is thus how the dependency on the parameters of interest is presented).

By contrast, Theorem~\ref{thm_main:Wpcontract} and Corollary~\ref{Cor:r*P} give explicit quantitative convergence rates for a \emph{discrete-time} Markov chain obtained as a discretization of \eqref{Eq:Continu-Langevin}. \nv{As mentioned in the introduction, this leads to the concentration inequalities of Theorem~\ref{thm_main:concentration} and it was motivated by the work \cite{QinHobert}}. That being said, in \cite{Chatterji1,dalalyan1,Zajic,dalalyan2,Chatterji2,Chatterji3,Zajic}, a discretization error analysis is conducted which, together with the convergence rate of the continuous-time process, provides \emph{in fine} non-asymptotic bounds   for the discrete-time algorithm. In the smooth and convex case, up to some logarithmic terms, $\W_2$ efficiency bounds of order $\sqrt d / \varepsilon$ are obtained in  \cite{Chatterji1,Zajic}  with a first-order approximation of \eqref{Eq:Continu-Langevin} and of order $\sqrt{d/\varepsilon}$ in \cite{dalalyan1}  with  a second-order scheme (that requires the computation of $\na^2 U$ at each iteration) under the Lipschitz Hessian condition (that corresponds to \refordren\ with $\ell=2$). These scalings are thus similar to our results on the OBABO chain under similar conditions. In \cite{Chatterji3}, using a second order scheme with a similar complexity of the OBABO chain, efficiency bounds of order $\sqrt{d/\varepsilon}$ are obtained for the relative entropy (Kullback-Leibler divergence) under the Lipschitz Hessian condition. Through Pinsker's and Talagrand's inequalities, this yields $\W_2$ and total variation efficiency bounds of order  $\sqrt{d}/\varepsilon$ \nv{(since the results have to be applied with $\varepsilon$ replaced by $\varepsilon^2$)}, to compare to our results respectively of $\sqrt{d/\varepsilon}$ and $d\nv{^{3/4}}/\sqrt\varepsilon$ for these distances under \refordren\ with $\ell=2$. \nv{Thus, for the unadjusted process, under the same conditions as \cite{Chatterji3}, we see that we have a better dependency in term of $\varepsilon$ for both Wasserstein and total variation distances, the same dependency in the dimension for the Wasserstein distance but a worse dependency in the dimension ($d^{3/4}$ versus $\sqrt d$) for the total variation distance. In the separable case we recover the $\sqrt d$ scaling for the total variation distance, and improve the scaling to $d^{1/4}$ for the Wasserstein one. Our method is very different from the one of \cite{Chatterji3}. Comparing the discrete scheme with the continuous process, they get a bound on the relative entropy that scales as $d$, which gives $\sqrt d$ for the total variation. We get a bound on the total variation by comparing the unadjusted process with the Metropolis-adjusted one. We can see that the bad dependency in $d$ for the total variation in our case stems from the fact we bound  the expected Metropolis rejection probability by a cubic term (Lemma~\ref{lem:alpha} with $\ell=2$) whose expectation is of order $d^{3/2}$. The bounds of Lemma~\ref{lem:alpha} seems to be sharp, so it is the method of comparing the unadjusted and adjusted processes which leads to a non-optimal scaling. In some sense, by requiring the two processes to stay equal, we do not use the regularization properties of the process. By working with the relative entropy, the law of the unadjusted process can be compared with the law of the continuous time process: the processes (driven by the same Brownian motion) do not stay equal but remain close which, with the regularization properties of the process, is in fact sufficient to control the relative entropy of the laws (roughly speaking), hence the total variation. }

Similarly to the results on the convergence rate of the Langevin diffusion, several works are concerned with the convergence rates of \emph{idealized} HMC algorithms, i.e. algorithms that rely on the exact simulation of the continuous-time Hamiltonian dynamics, \cite{idealHMC1,idealHMC2,DoucetHMC,Mangoubi,EberleHMC}.
 Non-asymptotic bounds on the unadjusted HMC are obtained in \cite{Mangoubi}, of order $\sqrt d /\varepsilon$ for the $\W_1$ distance in the smooth and convex case with a first-order scheme, of order $d^{1/4}/\sqrt\varepsilon$ in the separable case  (and later in \cite{Mangoubi2} under a weaker condition), and similarly for the Metropolis-adjusted algorithm. Again, this is similar to our rates for the OBABO chain.
 
 The works \cite{Dwivedi,Dwivedi2} are concerned with Metropolis-adjusted algorithms (Metropolis-adjusted overdamped Langevin algorithm -- MALA -- and HMC), and establish total variation efficiency bounds that are logarithmic in the accuracy $\varepsilon$. The method is quite different to ours, based on conductance bounds. In term of number of computations of gradients, the scalings are $d$ for the MALA algorithm and $d^{11/12}$ for the HMC one in the smooth, convex, Lipschitz Hessian case. In term of dependency in $d$, for the Metropolis-adjusted case, we only get $d$ in this case, and $\sqrt d$ in the separable case. It would be interesting to use the methods of \cite{Dwivedi,Dwivedi2} for the OM(BAB)O chain (see also Remark~\ref{rmq:localCouplMetropolis}).

 For more considerations on the family of sampler based, like the HMC process and Langevin diffusion, on the Hamiltonian dynamics, we also refer to the recent preprint \cite{Song} and references therein.

\section{Study of the OBABO chain}\label{Sec:proofOBABO}

\subsection{Wasserstein contraction in the convex case}\label{sec:proofContract}

For $z=(x,v)\in \R^{2d}$, and $a,b>0$ with $b^2 < a$, we consider the Euclidean norm
\[\|z\|^2_{a,b} \ = \ |x|^2 + 2 b x\cdot v + a|v|^2\,.\]

For $x,v,g,g'\in \R^d$, denote
\begin{eqnarray*}
\Theta_1(x,v,g,g') & = & x+ \delta \po \eta v + \sqrt{1-\eta^2} g\pf  - \frac{\delta^2}2 \na U(x)\\
\Theta_2(x,v,g,g') & = & \eta^2 v - \frac{\delta\eta }{2}\po \na U(x) + \na U\po \Theta_1(x,v,g,g')\pf\pf + \sqrt{1-\eta^2} \po \eta g +  g'\pf\,,
\end{eqnarray*}
and $\Theta=(\Theta_1,\Theta_2)$. The transition~\eqref{Eq:ABOBA} of the OBABO chain is thus given by $(x_1,v_1) = \Theta(x_0,v_0,G,G')$ with independent $G,G'\sim \mathcal N(0,\Id)$. We consider two initial conditions $(x_0,v_0)$, $(y_0,w_0)\in\R^d\times\R^d$, two independent sequences of independent standard Gaussian variables $(G_n)_{n\in\N}$ and $(G'_n)_{n\in\N}$, and the parallel coupling of two OBABO chains given by
\begin{equation}\label{Eq:couplage_parallele}
\forall n\in \N\,,\qquad (x_{n+1},v_{n+1}) = \Theta(x_n,v_n,G_n,G_n')\,,\qquad (y_{n+1},w_{n+1}) = \Theta(y_n,w_n,G_n,G_n')\,.
\end{equation} 

\begin{prop}\label{prop:contraction}
Under Assumption~\refsmooth and \refconvex, suppose moreover that \nv{$\gamma \geqslant 2\sqrt{L}$ and that $\delta \leqslant m/(33\gamma^3)$. Set 
\[a\ =\ 1/L\,,\qquad  b\ =\ 1/\gamma\,, \qquad \kappa \ = \  m/(3\gamma)\,.\]
}
Then, $b^2 \leqslant a/4$ and 
 for all $(x_0,v_0),(y_0,w_0)\in \R^{2d}$, the parallel coupling given by \eqref{Eq:couplage_parallele} is such that almost surely, for all $n\in\N$,
\[\|(x_n,v_n)-(y_n,w_n)\|_{a,b}^2 \ \leqslant \ \po 1 - \delta \kappa \pf^n \|(x_0,v_0)-(y_0,w_0)\|_{a,b}^2\,.\]
\end{prop}

\begin{rmq}
It is possible to improve the condition on $\gamma$ and the other estimates by assuming that $U(x) = x \cdot S^{-1}x + \tilde U(x)$ for some symmetric positive matrix $S$ and then taking $S$ into account in the definition of the modified norm with more care. In particular in the Gaussian case ($\tilde U = 0$) there is no condition on $\gamma$, it is always possible to design a Euclidean norm that is contracted by the coupling for sufficiently small $\delta$\nv{, see e.g. \cite{MonmarcheContraction}}. Nevertheless, in this work we are mainly concerned with the fact that $\kappa$ is independent from the dimension. 
\end{rmq}

\begin{rmq}
The contraction is almost sure, and not only in expectation. In fact, as will be clear in the proof, the particular law of $(G,G')$ does not intervene, it is sufficient that it does not depend on the position $(x,v)$.
\end{rmq}

\begin{proof}
Let us show that, for this choice of parameters, for all $x,v,y,w,g,g' \in \R^d $, denoting $(x',v')=\Theta (x,v,g,g')$ and $(y',w')=\Theta(y,w,g,g')$,
\[\| (x',v')-(y',w')\|_{a,b}^2 \ \leqslant \ \po 1 - \delta\kappa\pf \|(x,v)-(y,w)\|_{a,b}^2\,,\]
which will immediately yield the result.

As a first step, we identify the leading terms (in term of $\delta$) of the evolution of the norm. The contraction will be proven only thanks to these terms, the higher order ones being treated as a small perturbation. As a consequence, our arguments works for other Markov chains that are an approximation of order at least one (in $\delta$) of the continuous-time Langevin diffusion. Denote $\bx = x-y$, $\bx'=x'-y'$, $\bv = v-w$, $\bv'=v'-w'$ and $\Delta \bar x = \bx'-\bx$. \nv{Let $\xi,\xi'\in\R^d$ be such that $\na U(x')-\na U(y')=\na^2 U(\xi') \bx'$ and $\na U(x)-\na U(y)=\na^2 U(\xi) \bx$, and let $Q=\na^2 U(\xi)$, $Q'=\na^2 U(\xi')$. Under Assumptions~\refsmooth\ and \refconvex, $mI_d\leqslant Q\leqslant LI_d$ in the sense of symmetric matrices, and similarly for $Q'$. With these notations,
\begin{align*}
\begin{pmatrix}
\bx'\\ \bv'
\end{pmatrix} &=
\begin{pmatrix}
\bx\\ \bv
\end{pmatrix}  + \begin{pmatrix}
0 & \delta \eta I_d \\ -\frac{\delta\eta}2(Q+Q')  & (\eta^2-1)I_d 
\end{pmatrix} 
\begin{pmatrix}
\bx\\ \bv
\end{pmatrix} 
- \frac{\delta}{2}
\begin{pmatrix}
\delta Q \bx \\ \eta Q'\Delta \bar x
\end{pmatrix} = (I_{2d}+\delta A)\begin{pmatrix}
\bx\\ \bv
\end{pmatrix} + h
\end{align*}
with
\[A = \begin{pmatrix}
0 &     I_d \\ -\frac{1}2(Q+Q')  & -\gamma I_d 
\end{pmatrix} \,, \qquad h = \begin{pmatrix}
0 & \delta (\eta-1) I_d \\ -\frac{\delta(\eta-1)}2(Q+Q')  & (\eta^2-1+\delta \gamma)I_d 
\end{pmatrix} \begin{pmatrix}
\bx\\ \bv
\end{pmatrix} 
- \frac{\delta}{2}
\begin{pmatrix}
 \delta Q \bx \\ \eta Q'\Delta \bar x
\end{pmatrix}\,.\]
Writing $\bar z=(\bx,\bv)$, $\bar z'=(\bx',\bv')$ and
\[M = \begin{pmatrix}
I_d & b I_d \\ b I_d & a I_d 
\end{pmatrix}\,,\]
we get
\begin{align*}
\|\bar z'\|_{a,b}^2   \ &= \     \|\bar z\|_{a,b}^2 + \delta\bar z \cdot  (MA+A^TM)  \bar z + 2 \bar z\cdot M h + \|\delta A\bar z + h\|_{a,b}^2\\
&\leqslant \ \|\bar z\|_{a,b}^2 + \delta\bar z \cdot  (MA+A^TM)  \bar z + 2 \|\bar z\|_{a,b}\|h\|_{a,b} + 2\delta^2\| A\bar z\|_{a,b}^2 +2\| h\|_{a,b}^2\,.
\end{align*}
The choice $a=1/L$ and $b=1/\gamma$ with the condition $\gamma\geqslant 2\sqrt{L}$ ensures that $b^2\leqslant a/4$ and thus for all $z\in \R^d$,
\begin{align}\label{eq:equi_nome}
\frac12\|z\|_{a,0}^2 \leqslant \|z\|_{a,b}^2\leqslant\frac32\|z\|_{a,0}^2\,.
\end{align}
Using \eqref{eq:equi_nome} and the bounds $|Q|,|Q'|\leqslant L$, $|1-\eta|\leqslant\delta\gamma/2$, 
 $|1-\eta^2-\delta\gamma|\leqslant\delta^2\gamma^2/2$ and
\[|\Delta \bx| = |\delta  \eta \bv   -  \delta^2/2 Q\bx|\leqslant \delta|\bv|+\delta^2L/2|\bx|\]
 we roughly bound
\begin{align*}
\|h\|_{a,b}^2 &\leqslant \frac32\left|\delta (\eta-1) \bv-\frac{\delta^2}{2}Q\bx\right|^2 +\frac32 a\left|-\frac{\delta(\eta-1)}2(Q+Q')\bx+ (\eta^2-1+\delta\gamma)\bv -\frac{\delta\eta}{2}Q'\Delta\bx\right|^2\\
&\leqslant 3 \delta^4\gamma^2/4 |\bv|^2+  3\delta^4L^2/4|\bx|^2+3a \po\delta^2\gamma L/2+\delta^3L^2/4\pf^2|\bx|^2 +3a\po\delta^2\gamma^2/2+\delta^2L/2\pf^2|\bv|^2\\
&\leqslant 6\delta^4\max\po\frac{\po L\gamma/2+\delta L^2/4\pf^2}{L}+\frac{3L^2}{4},\frac{L\gamma^2}4+\po\frac{\gamma^2}2+\frac{ L}2\pf^2\pf\|\bar z\|_{a,b}^2 \\
&\leqslant 3\delta^4 \gamma^4 \|\bar z\|_{a,b}^2 
\end{align*}
where we used that $L \leqslant \gamma^2/4$ and $\delta^2 L \leqslant L^3/\gamma^6\leqslant 1$ to simply the expression,
and similarly
\[
\| A\bar z\|_{a,b}^2 \  \leqslant\ \frac32|\bv|^2+ \frac32 a \left|-\frac12(Q+Q')\bx -\gamma\bv\right|^2
 \ \leqslant\   3L|\bx|^2 + 3a\po\frac{L}{2}+\gamma^2\pf|\bv|^2 \ \leqslant \ \frac{27\gamma^2}8 \| \bar z\|_{a,b}^2 \,,
\]
 so that, using that $6\delta^2\gamma^2 \leqslant 1/2$,
 \[2 \|\bar z\|_{a,b}\|h\|_{a,b} + 2\delta^2\| A\bar z\|_{a,b}^2 +2\| h\|_{a,b}^2 \ \leqslant \ 11 \delta^2\gamma^2   \| \bar z\|_{a,b}^2 \,.\]
 On the other hand,
 \[MA+A^TM = \begin{pmatrix}
 -b(Q+Q')  &-a(Q+Q')/2 \\ -a (Q+Q')/2 & 2(b-a\gamma)I_d
\end{pmatrix} \,. \] 
Bounding $2\bx \cdot (Q+Q')\bv \leqslant \bx \cdot (Q+Q')\bx/\theta + \theta\bv\cdot (Q+Q')\bv$ with $\theta=\gamma/L$ we obtain 
\begin{align*}
\bar z \cdot (MA+A^TM) \bar z & \leqslant \bar z \begin{pmatrix}
 \po - \frac1\gamma+ \frac{1}{2L\theta}\pf (Q+Q')  & 0 \\ 0  & 2(b-a\gamma)I_d + \frac{a\theta}{2}(Q+Q')
\end{pmatrix} \bar z\\
& \leqslant - \frac{m}{\gamma}|\bx|^2 - a \gamma  \po 1-\frac{2L}{\gamma^2 }\pf|\bv|^2 \\
&\leqslant -\frac{2m}{3\gamma} \|\bar z\|_{a,b}^2\,,
\end{align*}
where we used that $\gamma/2 \geqslant L/\gamma \geqslant m/\gamma$ and \eqref{eq:equi_nome}. Finally, we have obtained that
\[\|\bar z'\|_{a,b}^2 \leqslant \po 1 - \frac{2m}{3\gamma}\delta + 11 \delta^2 \gamma^2\pf \|\bar z\|^2_{a,b} \ \leqslant \ \po 1 - \frac{m}{3\gamma}\delta \pf \|\bar z\|^2_{a,b}\,. \]
 
}

\end{proof}

For $p\geqslant 1$ and $a,b>0$ with $b^2\leqslant a/4$, denote by $\mathcal W_{p,a,b}$ the $\mathcal W_p$ Wasserstein distance on $\mathcal P_p\po \R^{2d}\pf$ associated to the distance $\|\cdot\|_{a,b}$, i.e.
\[\W_{p,a,b}\po \nu_1,\nu_2\pf \ = \ \inf_{\pi \in \Gamma(\nu_1,\nu_2)} \po \int_{\R^{2d}\times\R^{2d}} \| z-z'\|_{a,b}^p \pi(\dd z \dd z')\pf^{1/p}\,. \]
The equivalence of $\|\cdot\|_{a,b}$ with the Euclidean distance implies that 
\begin{equation}\label{eq:equivalenceWp}
\frac{1}2\min(1,a) \W_p^2 \ \leqslant \ \frac12 \W_{p,a,0}^2 \ \leqslant \ \W_{p,a,b}^2 \ \leqslant \ \frac32 \W_{p,a,0}^2 \ \leqslant \ \ \frac{3}{2}\max(1,a)\W_p^2\,.
\end{equation}
From this equivalence and Proposition~\ref{prop:contraction}, it is then straightforward to obtain the following, which proves Theorem~\ref{thm_main:Wpcontract}.

 \begin{cor}\label{Cor:Wpcontract}
 Under the assumptions and with the notations of Proposition~\ref{prop:contraction},
  for all $p\geqslant 1$ and all $\nu,\mu\in\mathcal P_p(\R^{2d})$, 
 \begin{eqnarray}\label{Eq:W2contract}
 \mathcal W_{p,a,b}^2\po \nu P,\mu P\pf & \leqslant  & (1-\delta\kappa) \mathcal W_{p,a,b}^2\po \nu ,\mu \pf\,,
 \end{eqnarray}
 and for all $n\in \N$,
 \[\W_p^2\po \nu P^n,\mu P^n\pf \ \leqslant \ \nv{3\max\po L,\frac1L\pf}  (1-\delta \kappa)^{n} \W_p^2(\nu,\mu)\,.\]
 Moreover, $P$ admits a unique invariant probability distribution $\pi_\delta$, and $\pi_\delta\in  \mathcal P_p(\R^{2d})$ for all $p\geqslant 1$. 
 \end{cor}
 
 \begin{proof}
 Consider the parallel coupling \eqref{Eq:couplage_parallele}  with $\po (x_0,v_0),(y_0,w_0)\pf$ a $\W_p$-optimal coupling  of $\nu $ and $\mu$ (which exists from  \cite[Corollary 5.22]{VillaniOldNew}). Then $z_n=(x_n,v_n)\sim \nu P^n$ and $z_n'=(y_n,w_n)\sim \mu P^n$, so that
 \[ \mathcal W_{p,a,b}^p\po \nu P^n,\mu P^n\pf \  \leqslant  \ \mathbb E \po \|z_n-z_n'\|_{a,b}^p \pf \ \leqslant \ (1-\delta\kappa)^{np/2}\mathbb E \po \|z_0-z_0'\|_{a,b}^p \pf \ = \ (1-\delta\kappa)^{np/2} \mathcal W_{p,a,b}^p\po \nu  ,\mu  \pf \,.\]
 Conclusion follows from the distance equivalence \eqref{eq:equivalenceWp} and the Banach fixed point theorem (recall $(\mathcal P_p(\R^{2d}),\W_p)$ is a Banach space,  \cite[Theorem 6.18]{VillaniOldNew}).
 \end{proof}
  
  \begin{rmq}
The one-step contraction \eqref{Eq:W2contract} implies that $P$ has a positive Wasserstein curvature $\delta\kappa/2$ in the sense of \cite{Ollivier,Joulin} with respect to the metric $\|\cdot\|_{a,b}$.
\end{rmq}
 
\subsection{Wasserstein/total variance regularization}\label{Sec:proofWTV}

While the proof of \cite[Corollary 4.7]{GuillinWang}  for the continuous-time process relies on functional inequalities arguments, our proof will be a direct application of optimal coupling for Gaussian variables, in the spirit of \cite{durmus2019,MadrasSezer}. 

As can be seen by conditioning with respect to the initial condition (distributed according to an optimal $\W_1$ coupling), Proposition~\ref{prop_main:W/TV} is an immediate corollary of the following.

\begin{prop}\label{prop:W/TV}
Under Assumption~\refsmooth, for all $(x_1,v_1),(x_2,v_2)\in\R^{2d}$,
\[\| P \po (x_1,v_1),\cdot\pf - P\po (x_2,v_2),\cdot\pf \|_{TV} \ \leqslant \ \frac{3\delta|v_1-v_2| + \po 2 + 3\delta^2 L/2\pf |x_1-x_2|}{\delta\sqrt{2\pi(1-\eta^2)}}\,.\]
\end{prop}

\begin{proof}
As already mentioned, for a fixed $(x,v)$, $(x',v') \sim P((x,v),\cdot)$ is not a Gaussian variable. Nevertheless, we are going to use successively two Gaussian couplings, coupling first the positions thanks to the random variable $G$ and then the velocities with the random variable $G'$. In particular, if the first coupling is a success, then the non-Gaussian part $\na U(x')-\na U(y')$ vanishes. This two steps coupling (position first, then velocity) is natural in view of the hypoelliptic structure of the Langevin diffusion. Note that the choice of the OBABO sampler is particularly convenient here, since the damping part is splited in two half-steps, contrary to e.g. the BAOAB scheme \cite{Leimkuhler} (although one would just have to consider two transitions of the Markov chain).

Recall the following result for the optimal coupling of Gaussian distributions. For $x,y\in\R^d$ and $\sigma>0$,
\begin{equation*} 
\|\mathcal N(x,\sigma^2 \Id) - \mathcal N(y,\sigma^2 \Id)\|_{TV} \ \leqslant \ \frac{|x-y|}{\sqrt{2\pi\sigma^2}}\,.
\end{equation*}
Indeed, orthogonal coordinates being independent, we just have to couple the projections on $x-y$, so that the $d$ dimensional case boils down to the one dimensional case, which is \cite[Lemma 15]{MadrasSezer}. 
 
Let $(x_1,v_1),(x_2,v_2)\in\R^d\times\R^d$. Denote 
\[b_i = x_i+\delta \eta v_i - \delta^2/2\na U(x_i)\]
 for $i=1,2$,  and let $(x_1',x_2')$ be an optimal coupling of $\mathcal N(b_i,\delta^2(1-\eta^2)\Id)$, $i=1,2$. Set $G_i = (x_i' - b_i)/(\delta \sqrt{1-\eta^2})$ for $i=1,2$. In other words, $G_1$ and $G_2$ are both standard Gaussian random variables and such that
\begin{equation}\label{xi'}
x_i' \ = \ b_i + \delta \sqrt{1-\eta^2} G_i\,,\qquad i=1,2
\end{equation} 
and
\begin{equation}\label{pxi'}
2\mathbb P \po  x_1' \neq x_2'\pf \ \leqslant \   \frac{|b_1-b_2|}{\delta \sqrt{2\pi(1-\eta^2)}}\,.
\end{equation}
Conditionally to $(x_1',x_2')$, denote 
\[c_i \ = \ \eta^2v_i - \frac{\delta\eta }{2}\po \na U(x_i) + \na U(x_i')\pf + \sqrt{1-\eta^2}   \eta G_i  \]
and let $(v_1',v_2')$ be an optimal coupling of $\mathcal N(c_i,(1-\eta^2)\Id)$, $i=1,2$, so that
\begin{equation}\label{pvi'}
2\mathbb P \po  v_1' \neq v_2'\ |\ (x_1',x_2')\pf \ \leqslant \   \frac{|c_1-c_2|}{\sqrt{2\pi(1-\eta^2)}}\,.
\end{equation}
Set $G_i' = (v_i'-c_i)/\sqrt{1-\eta^2}$ for $i=1,2$, so that $G_1'$ and $G_2'$ are both   standard Gaussian random variables such that 
\begin{equation}\label{vi'}
v_i' \ = \ c_i + \sqrt{1-\eta^2} G_i'\,,\qquad i=1,2.
\end{equation}
The fact that, conditionally to $(x_1',v_1')$, $G_1'$ is a standard Gaussian random variable, implies that $G_1'$ is independent from $(G_1,G_2)$ (and similarly for $G_2'$). As a consequence, \eqref{xi'} and \eqref{vi'} with $G_1$ and $G_1'$ (resp. $G_2$ and $G_2'$) that are two independent standard Gaussian variables proves that $(x_i',v_i')\sim P\po (x_i,v_i),\cdot\pf$ (recall the transition \eqref{Eq:ABOBA} corresponding to $P$). In particular,
\[\|\delta_{(x_1,v_1)} P - \delta_{(x_2,v_2)} P\|_{TV} \ \leqslant \ 2 \mathbb P \po (x_1',v_1')\neq(x_2',v_2')\pf\,.\]

We now bound this probability. Remark that
\[x_1'=x_2' \qquad \Leftrightarrow \qquad \sqrt{1-\eta^2}(G_1-G_2) = \frac{b_2-b_1}{\delta}\,.\]
Hence, under the event $\{x_1'=x_2'\}$, 
\[c_1-c_2 = \eta^2(v_1-v_2) - \frac{\delta\eta }{2}\po \na U(x_1) - \na U(x_2)\pf + \frac{\eta}{\delta}(b_2-b_1)\,. \]
Since
\[|b_1-b_2| \ \leqslant \ \po 1+\delta^2L/2\pf |x_1-x_2| + \delta |v_1-v_2|\,,\]
then, under $\{x_1'=x_2'\}$,
\[|c_1-c_2| \ \leqslant \ 2|v_1-v_2| + \po 1/\delta + \delta L\pf |x_1-x_2|\,.\]
As a conclusion, using  \eqref{pxi'} and \eqref{pvi'},
\begin{eqnarray*}
2\mathbb P \po (x_1',v_1')\neq(x_2',v_2')\pf & \leqslant & 2\mathbb P \po x_1'\neq x_2'\pf + 2\mathbb P \po v_1'\neq v_2'\ |\ x_1'=x_2'\pf\\
& & \\
& \leqslant & \frac{3\delta|v_1-v_2| + \po 2 + 3\delta^2 L/2\pf |x_1-x_2|}{\delta\sqrt{2\pi(1-\eta^2)}}\,.
\end{eqnarray*}
\end{proof}

\subsection{Concentration inequality}

We first state the following general result (proven in \cite[Theorem 1]{MalrieuTalay}  and  \cite[Corollary 3.5]{ChafaiMalrieuParoux} in particular cases, see also  \cite[Lemma 3.2]{MonmarchePDMP}).
\begin{thm}\label{thm:logSobgeneral}
Let $R$ be a Markov operator on $\R^d$ with invariant measure $\nu$. Suppose that there exists $r\geqslant 0$ such that, for all bounded Lipschitz function $\varphi$ and all $z\in\R^d$,
\[|\na R\varphi(z)| \ \leqslant \ r R|\na \varphi|(z).\]
Suppose moreover that there exists $C>0$ such $R(z,\cdot)$ satisfies a log Sobolev inequality with constant $C$ for all $z\in\R^d$. Then:
\begin{enumerate}
\item If $\mu\in\mathcal P(\R^d)$ satisfies a log Sobolev with constant $C'$, then $\mu R$ satisfies a log Sobolev inequality with constant $r^2C'+C$.
\item If $r<1$, then for all $n\in\N$ and all $z\in \R^d$, $R^n(z,\cdot)$ satisfies a log Sobolev inequality with constant $C(1-r^{2n})/(1-r^2)$, and $\nu$ satisfies a log Sobolev inequality with constant $C/(1-r^2)$.
\item If $r<1$, $\mu\in\mathcal P(\R^d)$ satisfies a log Sobolev with constant $C'$ and $(Z_n)_{n\in\N}$ is a Markov chain associated to $R$ with initial condition $Z_0\sim\mu$, then for all $n\in\N_*$, all $u\geqslant 0$,  and all $1$-Lipschitz functions $\varphi$,
\begin{eqnarray*}
\mathbb{P}\po \frac1n\sum_{k=1}^n  \po \phi(Z_k) - \mathbb E \po \phi(Z_k)\pf\pf \geqslant u\pf & \leqslant &   \exp \po - \frac{n u^2(1-r)^2}{C+rC'/n}\pf
\end{eqnarray*}
and
\[\left| \frac1n\sum_{k=1}^n    \mathbb E \po \phi(Z_k)\pf  - \nu(\phi)\right| \ \leqslant \ \frac{r}{n(1-r)}\W_1\po \nu,\mu\pf\,.\]
\end{enumerate}

\end{thm}

\begin{rmq}\label{rmq:logSobLip}
The first point generalizes the fact that log-Sobolev inequalities are transported by Lipschitz transformations, which corresponds to the case $R\phi = \phi \circ\Psi$ with $\Psi$ an $r$-Lipschitz function (in that case, $C=0$).
\end{rmq}

\begin{proof}
By density, we can consider a smooth positive Lipschitz $f$. Denoting $g = \sqrt{\nv{R} f^2}$, we decompose
\begin{eqnarray*}
\mu \nv{R} \po f^2 \ln f^2\pf - \mu \nv{R}(f^2) \ln \mu \nv{R}(f^2) &=  & \mu \po \nv{R}(f^2\ln f^2) - \nv{R}(f^2)\ln \nv{R} (f^2)\pf \\
& & \qquad +\  \mu(g^2 \ln g^2) - \mu(g^2) \ln \mu(g^2)\\
& \leqslant & C \mu \nv{R}\po |\na f|^2\pf + C'\mu \po |\na g|^2\pf \,.
\end{eqnarray*}
Now, by the Cauchy-Schwarz inequality,
\[|\na g|^2 \ = \ \frac{  |\na \nv{R} (f^2)| ^2}{4\nv{R} (f^2)} \ \leqslant  \ \frac{r^2\po \nv{R}|\na (f^2)|\pf^2}{4\nv{R} (f^2)} \ \leqslant  \ \frac{r^2}4 \nv{R} \po \frac{ |\na (f^2)|^2}{f^2} \pf \ = \  r^2 \nv{R}\po |\na f|^2\pf\,, \]
which concludes the proof of the first claim. The second claim immediately follows by induction and the weak convergence of $\nv{R}^n$ toward $\nu$, which follows from the Wasserstein contraction implied by the gradient/operator subcommutation, see \cite[Theorem 2.2]{Kuwada}.

Concerning the third claim, if $\mu\in\mathcal P(\R^d)$ satisfies a log-Sobolev inequality with constant $C'>0$ then, for all $\beta$-Lipschitz $\varphi$ and all $\lambda$,
\begin{equation}\label{eq:logSobexpo}
\mu \po  e^{\lambda\varphi}\pf \ \leqslant \ e^{C'\beta^2\lambda^2 /4} e^{\mu(\varphi)}\,,
\end{equation}
see e.g. \cite{LedouxConcentration}. Using successively Chernov's inequality, the Markov property and \eqref{eq:logSobexpo} together with the local log-Sobolev inequality satisfied by $R$, given any $1$-Lipschitz $\phi$, $\lambda> 0$ and $t\in\R$, 
\begin{eqnarray*}
\mathbb{P}\po \frac1n\sum_{k=1}^n \phi(Z_k) \geqslant t\pf & \leqslant & e^{-nt\lambda}\mathbb E \po e^{\lambda\sum_{k=1}^n \phi(Z_k)}\pf\\
& = &  e^{-nt\lambda}\mathbb E \po e^{\lambda\sum_{k=1}^{n-1} \phi(Z_k)}R \po e^{\lambda \phi}\pf(Z_{n-1})\pf \\
& \leqslant & e^{-nt\lambda+C\lambda^2/4}\mathbb E \po e^{\lambda\sum_{k=1}^{n-2} \phi(Z_k)}  e^{\lambda (\phi+R\phi)(Z_{n-1})} \pf \\
& = & e^{-nt\lambda+C\lambda^2/4}\mathbb E \po e^{\lambda\sum_{k=1}^{n-2} \phi(Z_k)}  R\po e^{\lambda (\phi+R\phi)}\pf(Z_{n-2}) \pf 
\end{eqnarray*}
Using that $\phi+R\phi$ is $(1+r)$-Lipschitz, \eqref{eq:logSobexpo} and then a straightforward induction, we get
\begin{eqnarray*}
\mathbb{P}\po \frac1n\sum_{k=1}^n \phi(Z_k) \geqslant t\pf & \leqslant &  e^{-nr\lambda+C\lambda^2\co 1+ (1+r)^2\cf  /4}\mathbb E \po e^{\lambda\sum_{k=1}^{n-2} \phi(Z_k)}  e^{\lambda (R\phi+R^2\phi)(Z_{n-2})} \pf \\
 &\leqslant &  e^{-nt\lambda+nC\lambda^2  /(4(1-r)^2)}\mathbb E \po e^{\lambda\sum_{k=1}^{n} R^k\phi(Z_0)} \pf \\
  &\leqslant &  e^{-nt\lambda+(nC+rC')\lambda^2  /(4(1-r)^2)} e^{\lambda\sum_{k=1}^{n} \mu R^k\phi}  .
\end{eqnarray*}
Applying this with $t = \sum_{k=1}^{n} \mu R^k\phi + u$ with $u\geqslant 0$ reads
\begin{eqnarray*}
\mathbb{P}\po \frac1n\sum_{k=1}^n \phi(Z_k) - \frac1n\sum_{k=1}^n \mu R^k \phi \geqslant u\pf & \leqslant & \inf_{\lambda>0}  \exp\po -n u\lambda+\frac{nC+rC'}{4(1-r)^2}\lambda^2  \pf \\
& = & \exp \po - \frac{n u^2(1-r)^2}{C+rC'/n}\pf\,.
\end{eqnarray*}
Using that $\nu$ is invariant by $R$, that $\sum_{k=1}^n R^k\phi$ is a $r/(1-r)$-Lipschitz function for all $n\geqslant 1$ and that the dual representation of  $\W_1$ (see e.g. \cite{VillaniOldNew}), 
\[\left|\frac1n \sum_{k=1}^n \mu R^k \phi - \nu \phi\right| \ = \ \frac1n  \left|(\mu-\nu) \po \sum_{k=1}^n R^k \phi\pf\right| \ \leqslant \ \frac{r}{n(1-r)}\W_1(\nu,\mu)\,. \]
\end{proof}

In the rest of this section, considering $a,b$ given by Proposition~\ref{prop:contraction},  we consider the matrix $S$ such that $S(x,v) = (x+bv,\sqrt{a-b^2}v)$. Hence, $\|(x,v)\|_{a,b} = |S(x,v)|$. We also denote $P_{a,b}$ the transition operator of $(S(x_n,v_n))_{n\in\N}$ where $(x_n,v_n)_{n\in\N}$ is an OBABO chain. This change of variable is meant for simplicity, in order to work with the classical Euclidean metric and gradient in the following.

\begin{lem}\label{lem:PablogSob}
Under the conditions of Proposition~\ref{prop:contraction}, for all Lipschitz functions $\varphi$ on $\R^{2d}$ and all $z\in\R^{2d}$,
\[|\na P_{a,b} \varphi(z)| \ \leqslant  \ \sqrt{1 - \delta \kappa}   P_{a,b}|\na \varphi|(z)\,.\]
Moreover, for all $z\in \R^{2d}$, $P_{a,b}(z,\cdot)$ satisfies a log-Sobolev inequality with constant 
\[C \ = \ 2|S|^2(1-\eta^2)\po 1 + \delta + \delta^2   L/2\pf^2  \,.\]
\end{lem}
\begin{proof}
Using the notations of Proposition~\ref{prop:contraction}, the latter means that,  setting $z_n=(x_n,v_n)$ and $z_n'=(y_n,w_n)$, 
\[|Sz_{1}-Sz_{1}'| \ \leqslant \ \sqrt{1-\delta\kappa}|Sz_0-Sz_0'|\]
almost surely for all $z_0,z_0'$. This implies that $P_{a,b}$ contracts the $\W_\infty$ distances (for the classical Euclidean metric) by a factor $\sqrt{1-\delta\kappa}$, which implies the claimed gradient/Markov operator subcommuation as proven in \cite[Theorem 2.2]{Kuwada}.

For $z=(x,v)\in\R^{2d}$, $P_{a,b}(z,\cdot)$ is the image of $\mathcal N(0,I_{2d})$ by the transformation $\Psi:(G,G')\mapsto S\Theta(S^{-1}z,G,G')$. For $(g,g'),(h,h')\in\R^{2d}$,
\[|\Theta_1(z,g,g')-\Theta_1(z,h,h')| \ = \  \delta  \sqrt{1-\eta^2} |g-h|   \]
and thus
\begin{eqnarray*}
|\Theta_2(z,g,g')-\Theta_2(z,h,h')|  & \leqslant &  \frac{\delta\eta }{2} L |\Theta_1(z,g,g')-\Theta_1(z,h,h')|+ \sqrt{2(1-\eta^2)} \po \eta |g-h| +  |g'-h'|\pf\\
 & \leqslant &  \frac{\delta\eta }{2} L \delta  \sqrt{1-\eta^2} |g-h| + \sqrt{1-\eta^2} \po \eta |g-h| +  |g'-h'|\pf.
\end{eqnarray*}
As a consequence, $\Psi$ is Lipschitz with constant $|S|\po 1 + \delta + \delta^2   L/2\pf \sqrt{1-\eta^2}$. 
 Since the standard Gaussian distribution satisfies a log-Sobolev inequality with constant $2$, conclusion follows from Remark~\ref{rmq:logSobLip}.
\end{proof}

\begin{proof}[Proof of Theorem~\ref{thm_main:concentration}]
From Theorem~\ref{thm_main:Wpcontract},  $\tilde \pi = S\sharp \pi_\delta $ is the unique invariant measure of $P_{a,b}$. From Lemma~\ref{lem:PablogSob} and Theorem~\ref{thm:logSobgeneral}, $\tilde \pi$ satisfies a log-Sobolev inequality with constant $ C/(\delta\kappa)$.  From Remark~\ref{rmq:logSobLip}, $\pi_\delta = S^{-1}\sharp \tilde \pi $ then satisfies a log-Sobolev inequality with constant $|S^{-1}|^2 C/(\delta\kappa)$. Besides, $|S||S^{-1}|= K_1$ and $(1-\eta^2)/\delta \leqslant \gamma$, which gives the claimed expression for the log-Sobolev constant of $\pi_\delta$.

For $\phi$ a $1$-Lipschitz function, $\tilde \phi:=(\phi\circ S^{-1})/|S^{-1}|$ is $1$-Lipschitz. If $\mu\in\mathcal P(\R^{2d})$ satisfies a log-Sobolev inequality with constant $C'$ then $\mu \circ S$ satisfies an inequality with constant $|S|^2C'$. Hence, with $Z_0\sim \mu$, Theorem~\ref{thm:logSobgeneral}  applied to $P_{a,b}$ reads
\begin{eqnarray*}
\mathbb{P}\po \frac1n\sum_{k=1}^n  \po \phi(Z_k) - \mathbb E \po \phi(Z_k)\pf\pf \geqslant u\pf & = & \mathbb{P}\po \frac1n\sum_{k=1}^n  \po \tilde \phi(SZ_k) - \mathbb E \po \tilde\phi(SZ_k)\pf\pf \geqslant \frac{u}{|S^{-1}|}\pf \\ 
 & \leqslant &   \exp \po - \frac{n u^2(1-r)^2}{|S^{-1}|^2 \po C+r|S|^2C'/n\pf}\pf\,,
\end{eqnarray*}
with $r=\sqrt{1-\delta\kappa}\leqslant 1 -\delta\kappa/2$. Applying this to $\phi$ and $-\phi$ gives
\begin{eqnarray*}
\mathbb{P}\po \left|\frac1n\sum_{k=1}^n  \po \phi(Z_k) - \mathbb E \po \phi(Z_k)\pf\pf\right| \geqslant u\pf & \leqslant & 2 \exp \po - \frac{n u^2\delta^2 \kappa^2}{4|S^{-1}|^2 \po C+|S|^2C'/n\pf}\pf\,.
\end{eqnarray*}
Applying the last proposition of Theorem~\ref{thm:logSobgeneral} to $P_{a,b}$, conclusion follows from
\begin{eqnarray*}
\left| \frac1n\sum_{k=1}^n    \mathbb E \po \phi(Z_k)\pf  - \pi_{\delta}(\phi)\right|  & = & |S^{-1}| \left| \frac1n\sum_{k=1}^n    \mathbb E \po \tilde \phi( SZ_k)\pf  - \pi_{\delta}(\tilde \phi \circ S)\right| \\
& \leqslant & \frac{2|S^{-1}|}{n\kappa\delta}\W_1(S\sharp \mu, S\sharp  \pi_\delta) \ \leqslant \ \frac{2|S^{-1}||S|}{n\kappa\delta}\W_1(\mu, \pi_\delta) \,.
\end{eqnarray*}

\end{proof}

\section{Bias analysis}\label{Sec:CoupleVerlet}

First, we prove Proposition~\ref{prop_main:nupi1}, which relates the equilibrium error between $\pi$ and $\pi_\delta$ to the one-step error between $\pi$ and $\pi P_V$.

 \begin{proof}[Proof of Proposition~\ref{prop_main:nupi1}]
We use the notations of Proposition~\ref{prop:contraction}.    Using the  invariance of $\pi_\delta$ by $P$ and Corollary~\ref{Cor:Wpcontract},
  \begin{eqnarray*}
\mathcal W_{p,a,b}\po \pi_\delta ,\pi\pf   & \leqslant & \mathcal W_{p,a,b}\po \pi_\delta P ,\pi P\pf  + \mathcal W_{p,a,b}\po \pi P  ,\pi \pf \\
& \leqslant& (1-\delta\kappa/2) \mathcal W_{p,a,b}\po \pi_\delta ,\pi\pf +    \mathcal W_{p,a,b}\po \pi P  ,\pi  \pf
 \end{eqnarray*}
 where we used that $\sqrt{1-x} \leqslant 1-x/2$ for $x\in[0,1]$. As a consequence,
 \[\mathcal W_{p,a,b}\po \pi_\delta ,\pi\pf  \ \leqslant \ \frac{2}{\delta\kappa}\mathcal W_{p,a,b}\po \pi P  ,\pi  \pf \]
 which together with the equivalence of $|\cdot|$ and $\|\cdot\|_{a,b}$ (see \eqref{eq:equivalenceWp}) gives
 \[\mathcal W_{p}\po \pi_\delta ,\pi\pf  \ \leqslant \ \frac{2 K_1}{\delta\kappa}\mathcal W_{p}\po \pi P  ,\pi  \pf \]
 where $K_1$ is given in Theorem~\ref{thm_main:Wpcontract}.  Now, decomposing $P=P_OP_VP_O$, where $P_O$ has been introduced in Section~\ref{Sec:MOBABO},  we remark that  $P_O$ does not increase $\mathcal W_{p}$, as can be seen with a parallel coupling. Indeed, let $(x_i,v_i)\in\R^{2d}$, $i=1,2$,  and $G\sim \mathcal N(0,I_d)$. Set $x_i'=x_i$ and $v_i' = \eta v_i + \sqrt{1-\eta^2} G$ for $i=1,2$. Then $(x_i',v_i') \sim \delta_{(x_i,v_i)}P_O$ and
 \[|x_1'-x_2'|^2 +  |v_1'-v_2'|^2  \ \leqslant \ |x_1-x_2|^2 +  |v_1-v_2|^2\]
 almost surely. Considering $\nu_i \in \mathcal P_p(\R^d)$ for $i=1,2$, sampling $(x_i,v_i)\sim \nu_i$, using the previous coupling, taking the power $p/2$ and the expectation yields
 \[\W_{p}(\nu_1 P_O,\nu_2P_O) \ \leqslant \ \W_{p}(\nu_1 ,\nu_2)\,.\]
 Moreover, $\pi$ is invariant by $P_O$, so that
 \[\mathcal W_{p}\po \pi P  ,\pi  \pf \ = \ \mathcal W_{p}\po \pi P_OP_VP_O  ,\pi P_O  \pf \ \leqslant \ \mathcal W_{p}\po \pi P_V  ,\pi  \pf\,.\]
 For the total variation distance, similarly, for $n\geqslant 1$,
 \begin{eqnarray*}
 \|\pi_\delta - \pi\|_{TV} 
 & \leqslant & \|\pi_\delta P^n - \pi P^n\|_{TV} + \sum_{k=1}^n\|(\pi P-\pi) P^{k-1}  \|_{TV} \\
 & \leqslant &  \frac{K_1K_2 }{\delta^{3/2}}(1-\delta \kappa)^{(n-1)/2} \W_1(\pi_\delta , \pi )+ n \|\pi  P - \pi  \|_{TV} \,,
 \end{eqnarray*}
 where we used Corollary~\ref{Cor:r*P} and the fact the total variation is decreased by any Markov operator. Using again that $\pi$ is invariant by $P_O$,
 \[\|\pi  P - \pi  \|_{TV} \ = \ \|\pi  P_O P_V P_O - \pi P_O  \|_{TV} \ \leqslant \ \|\pi  P_V - \pi  \|_{TV}\,,\]
 which concludes.
  \end{proof}
  
As mentioned in Section~\ref{Sec:mainresBias}, in the following, we compare $P_V$ to two reference transitions: the first one is $P_{MV}$, the Metropolis-adjusted Verlet transition introduced in Section~\ref{Sec:MOBABO}, with acceptance probability $\alpha$. The second is $Q_\delta$ where $(Q_t)_{t\geqslant 0}$ is the deterministic transition operator of the Hamiltonian dynamics: given $(x,v)\in\R^{2d}$ and considering  $(\hat x_t,\hat v_t)_{t\geqslant 0}$  the solution of
\begin{equation}\label{HD}
\hat x_t' \ = \ \hat v_t \,,\qquad \hat v_t ' \ = \ -\na U(\hat x_t)\,,\qquad \hat x_0=x\,,\qquad \hat v_0=v\,,
\end{equation}
denote $\Phi_{HD}(x,v) = (x_\delta,v_\delta)$ and $Q_\delta\varphi(x,v) = \varphi(x_\delta,v_\delta)$.

  \begin{lem}\label{lem:piPV}
  For $\nu \in \mathcal P(\R^{2d})$ and $(X,V)\sim \nu$,
  \begin{eqnarray*}
  \| \nu P_V - \nu P_{MV}\|_{TV} & \leqslant &   2\mathbb E\po 1- \alpha(X,V)\pf \,.
  \end{eqnarray*}
If, moreover, $\nu \in \mathcal P_{p}(\R^{2d})$ for some $p\geqslant 1$,  then
   \begin{eqnarray*}
   \W_p\po \nu P_V,\nu P_{MV} \pf & \leqslant  &  \po \mathbb E\co \left|(X,-V)-\Phi_V(X,V)\right|^p (1-\alpha(X,V))\cf\pf^{1/p}
  \end{eqnarray*}
  and
     \begin{eqnarray*}
   \W_p\po \nu P_V,\nu Q_\delta \pf & \leqslant  &  \po \mathbb E\co \left| \Phi_V(X,V) - \Phi_{HD}(X,V) \right|^p  \cf\pf^{1/p}
  \end{eqnarray*}
 
  \end{lem}
   \begin{proof}
 The only difference between $P_V$ and $P_{MV}$ is the accept/reject step. Considering 
 $(X,V)\sim \nu$ and, independently, $W$ a random variable uniformly distributed over $[0,1]$, set
  \[(X',V') \ = \ (X,-V)\1_{W> \alpha(X,V)} + \Phi_V(X,V)\1_{W\leqslant \alpha(X,V)}\]
Then, $(X',V') \sim \nu P_{MV}$, while $\Phi_V(X,V) \sim \nu P_V$, so that 
\[ \| \nu P_V - \nu P_{MV}\|_{TV} \ \leqslant \   2\mathbb P \po (X',V')\neq \Phi_V(X,V)\pf \ =\   2\mathbb E\po 1- \alpha(X,V)\pf\,. \]
Similarly,
\begin{eqnarray*}
 \W_p^p\po \nu P_V,\nu P_{MV} \pf & \leqslant  &  \mathbb E \po |(X',V') - \Phi_V(X,V)|^p\pf \\
 & = &  2^p\mathbb E\po \left|(X,-V)-\Phi_V(X,V)\right|^p(1-\alpha(X,V))\pf\,. 
\end{eqnarray*}
Finally, the last claim follows from the fact $\po \Phi_V(X,V) , \Phi_{HD}(X,V)\pf$ is a coupling of $\nu P_V$ and $\nu Q_\delta$.

\end{proof}

\begin{rmq}
As $\delta \rightarrow 0$, $|(x,-v)-\Phi_V(x,v)| \rightarrow 2|v|$. When comparing a Metropolis adjusted Verlet step and an unadjusted one, in case of rejection, the distance is not small with $\delta$.
\end{rmq}

   \begin{lem}\label{lem:alpha}\
 
   \begin{enumerate}
\item Under Assumption~\refsmooth, for all $x,v\in \R^d$, writing $Z = |v|+\delta/2|\na U(x)|$,
\begin{eqnarray*}
1-\alpha(x,v) & \leqslant &   \delta^3 \frac18 |\na U(x)|^2   +  \delta^2 L\po \frac12+\delta \frac18+\delta^2\frac{L}8\pf Z^2 \,.
\end{eqnarray*}
\item Under Assumptions~\refsmooth\ and \refordren,  for all $(x,v) \in \R^{2d}$, writing $Z = |v|+\delta/2|\na U(x)|$,
\begin{eqnarray*}
1-\alpha(x,v) & \leqslant & \frac{\delta^3}8 \co  |\na U(x)|^2   +  (L+\delta L^2) Z^2+    L_\ell Z ^3   \po 1   + \po |x-x_\star|+\delta Z\pf^{\ell-2}\pf \cf\,.
\end{eqnarray*}
   \end{enumerate}

  \end{lem} 

  \begin{proof}

 Since $1-e^{-s}\leqslant s$ for all $s>0$,
\[   1- \alpha \ \leqslant \   \po H\circ \Phi_V - H\pf_+  \ \leqslant \    |H\circ \Phi_V- H| \,.\]
 For $x,v\in\R^d$, denoting $\xi = \delta v - \delta^2/2 \na U(x)$ and $\overline \na U = (\nabla U(x)+\nabla U(x+\xi))/2$,
 \begin{eqnarray*}
 H(\Phi_V(x,v)) - H(x,v) & = & U(x+\xi) - U(x) + \frac12 |v - \delta \overline  \na U|^2 - \frac12|v|^2 \\
 &  = & \xi\cdot \int_0^1 \na U(x+s\xi)\dd s -   \delta v \cdot \overline  \na U + \delta^2/2|\overline  \na U|^2\\
 & =& \xi\cdot \po \int_0^1 \na U(x+s\xi)\dd s - \overline \na U\pf + \delta^2/2 \overline  \na U \cdot \po \overline  \na U - \na U(x)\pf\\
 & =& \xi\cdot \int_0^1 (2g(s)-g(0)-g(1))\dd s + \delta^2/2 g(1) \cdot \po g(1)-g(0)\pf
 \end{eqnarray*}
 with $g(s) = (\na U(x) + \na U(x+s\xi))/2$ for $s\in[0,1]$.  Under Assumption~\refsmooth,
\[\left| g(1) \cdot \po g(1)-g(0)\pf\right| \ \leqslant \ \frac12 L |g(0)| |\xi| + \frac{1}{4}L^2 |\xi|^2 \ \leqslant \ \delta \frac L 4 \po  |\na U(x)|^2 + (1+L\delta )|\xi/\delta|^2\pf\,,\]
and similarly, for all $s\in[0,1]$,
\[|2g(s)-g(0)-g(1)| \ \leqslant \ \frac12 L |\xi|\,.\]
We have finally obtained, in the case where Assumption~\refsmooth\ alone is assumed,
 \begin{eqnarray*}
 |H(\Phi_V(x,v)) - H(x,v) |& \leqslant  & \delta^3 \frac L 8 \po  |\na U(x)|^2 + (1+L\delta )|\xi/\delta|^2\pf +  \frac12 L |\xi|^2
 \end{eqnarray*}
 and bounding $|\xi|\leqslant \delta |v| + \delta^2/2|\na U(x)|$ concludes the proof of the first claim.

However, under additional assumptions, we can treat more carefully the term
 \begin{eqnarray}\label{calculintg-g}
 2\int_0^1 g(s) \dd s - g(0)-g(1) & = & 2\int_0^1   \int_0^s g'(u) \dd u\dd s  -\int_0^1 g'(s)   \dd s \notag\\
 & = & \int_0^1 (1-2s) g'(s)\dd s \notag\\
 & = & \int_0^1 (1-2s) \po g'(s) - g'(0)\pf\dd s \,.
  \end{eqnarray}
Under Assumption~\refordren, as $\int_0^1 |1-2s|s\dd s  \leqslant 1/2$,
\begin{eqnarray*}
\left| \int_0^1 (1-2s) \po g'(s) - g'(0)\pf\dd s \right| & \leqslant & \frac1{16}L_\ell|\xi|^2   \po 2 + |x-x_\star|^{\ell-2} + \po |x-x_\star|+|\xi|\pf^{\ell-2}\pf \\
 & \leqslant & \frac1{8}L_\ell|\xi|^2   \po 1   + \po |x-x_\star|+|\xi|\pf^{\ell-2}\pf \,.
\end{eqnarray*}

  \end{proof}

\begin{lem}\label{lem:VerletDiscretization}\ 

\begin{enumerate}
\item Under Assumption \refsmooth, for all $(x,v) \in \R^{2d}$, writing $Z = |v|+\delta/2|\na U(x)|$,
\[|\Phi_V(x,v) - \Phi_{HD}(x,v)| \ \leqslant \   L \delta^2 \co 1 +  e^{L \delta^2 } \po \delta /6    + \delta^2 L/24\pf  \cf Z \,. \]
\item Under Assumptions \refsmooth\ and \refordren, for all $(x,v) \in \R^{2d}$, writing $Z = |v|+\delta/2|\na U(x)|$, 
\begin{multline*}
|\Phi_V(x,v) - \Phi_{HD}(x,v)| \ \leqslant \ \frac{\delta^3}4|\na U(x)| 
   +   \frac{L\delta^3}{6} \co        1 + \delta \frac{L+1}{2}  + \delta^3 \frac{L}{40}    \cf e^{L \delta^2 }  Z \\
   +\ \delta^3 \frac{L_\ell}{8}\po 1+ e^{L\delta^2} \frac{\delta^2}{12}\pf Z^2 \po 1 + \co |x-x_\star|  + \po \delta + e^{L\delta^2} \frac{\delta^3}{6}\pf Z \cf^{\ell-2}\pf\,.
\end{multline*}
\end{enumerate} 
\end{lem}
\begin{proof}
Fix $x,v\in\R^{2d}$. Assumption \refsmooth\ holds in the whole proof. Consider $(\hat x_t,\hat v_t)_{t\geqslant 0}$ given by \eqref{HD} and $(y_t,w_t)_{t\geqslant 0}$ the solution of 
\[y_t' \ = \ w_t \,,\qquad w_t ' \ = \ -\na U(x)\,,\qquad  y_0=x\,,\qquad  w_0=v\,.\]
In other words,
\[y_t \ = \ x + t v - \frac{t^2}{2}\na U(x)\,,\qquad w_t = v-t\na U(x)\,,\]
so that $\Phi_V(x,v) = (y_\delta,\tilde w_\delta)$ with $\tilde w_\delta = w_\delta + \delta/2(\na U(y_\delta)-\na U(x))$. Then,
\begin{eqnarray*}
|\hat x_t - y_t| & = & \left| \int_0^t \int_0^s (\na U(\hat x_u) - \na U(x)) \dd u \dd s \right|\\
& \leqslant & L \int_0^t \int_0^s \po |\hat x_u - y_u|+|y_u-x|\pf \dd u \dd s\\
& \leqslant & L \int_0^t \int_u^t\po |\hat x_u - y_u|+u|v|+u^2/2|\na U(x)|\pf  \dd s\dd u \\
& \leqslant & L \int_0^t (t-u) |\hat x_u - y_u|   \dd u + \frac{Lt^3}{6}|v|+ \frac{Lt^4}{24}|\na U(x)| \,.
\end{eqnarray*}
Bounding $(t-u)\leqslant t$ and applying Gronwall's Lemma, we get for all $t\in[0,\delta]$
\[|\hat x_t - y_t| \ \leqslant \ e^{L t^2 } \po \frac{L t^3}{6}|v|+\frac{L t^4}{24}|\na U(x)|\pf \ \leqslant \ e^{L t^2 } \frac{Lt^3}{6} Z \,.\]
Similarly, for all $t\in [0,\delta]$,
\begin{eqnarray*}
|\hat v_t - w_t| & = & \left|\int_0^t (\na U(\hat x_s) - \na U(x)) \dd u \right| \\
& \leqslant & L\int_0^t \po |\hat x_s - y_s| + |y_s-x| \pf \dd s  \\
& \leqslant&   e^{L t^2 } \frac{L^2t^4}{24} Z + \frac{L t^2}{2}Z\,.
\end{eqnarray*}
As a consequence,
\begin{eqnarray*}
 |\hat v_{\delta} -\tilde  w_\delta  | & \leqslant &  |\hat v_{\delta} -   w_\delta  | + \frac{\delta L}{2} |y_\delta - x |\\
  & \leqslant&  e^{L \delta^2 } \frac{L^2\delta^4}{24} Z +  L \delta^2 Z  \,.
\end{eqnarray*}
The conclusion of the proof of the first claim follows from the previous estimates and
\[|\Phi_V(x,v) - \Phi_{HD}(x,v)|  \ \leqslant \ |\hat x_\delta - y_\delta| + |\hat v_{\delta} -\tilde  w_\delta  |\,.\]

Note that $|\hat v_\delta - w_\delta|$ yields a first order error (in $\delta$). Now, if an additional assumption  is available, in order to see that the Verlet scheme is in fact a second order integrator, we should rather directly bound
\begin{eqnarray*}
 |\hat v_{\delta} -\tilde  w_\delta  |  
 & = & \left|\int_0^\delta \na U(\hat x_s)\dd s  - \delta \frac{\na U(x)+\na U(y_\delta)}{2}\right|\\
  & \leqslant  & \left|\int_0^\delta \na U(\hat x_s)\dd s  - \delta \frac{\na U(x)+\na U(\hat x_\delta)}{2}\right| + \frac{L\delta }{2}|y_\delta-\hat x_\delta|\,.
\end{eqnarray*}
As in the proof of Lemma~\ref{lem:alpha} (more precisely, \eqref{calculintg-g}), denoting $g(s) = \na U(\hat x_s)$,
\begin{eqnarray*}
\left|\int_0^\delta g(s)\dd s  - \delta \frac{g(0)+g(\delta)}{2}\right| & =& \left|\int_0^\delta  \po \frac\delta 2 - u\pf \po g'(s) - g'(0)\pf \dd s\right|\,,
\end{eqnarray*}
and
\begin{eqnarray*}
|g'(s)-g'(0)| & \leqslant & |  \po \na^2 U (\hat x_s) - \na^2 U(x) \pf v | + |\na^2 U(x) \po \hat v_s - v\pf|\\
& \leqslant & |  \na^2 U (\hat x_s) - \na^2 U(x)| | v | + L \po |\hat v_s - w_s|  + |w_s - v|\pf\,.
\end{eqnarray*}
Under Assumption~\refordren,
\begin{eqnarray*}
|  \na^2 U (\hat x_s) - \na^2 U(x)| & \leqslant & \frac{L_\ell}{2}\po |y_s-x|+|\hat x_s-y_s|\pf \po 1 + \po |x-x_\star|  + |y_s-x|+|\hat x_s-y_s|\pf^{\ell-2}\pf\\
 & \leqslant & \frac{L_\ell}{2}\po s+ e^{L\delta^2} \frac{s^3}{6}\pf Z  A_\ell\,,
\end{eqnarray*}
with
\[A_\ell \ = \   1 + \co |x-x_\star|  + \po \delta + e^{L\delta^2} \frac{\delta^3}{6}\pf Z \cf^{\ell-2} \,,\]
so that
\begin{eqnarray*}
\frac{\delta}{2}\int_0^\delta |v| |  \na^2 U (\hat x_s) - \na^2 U(x)|\dd s  & \leqslant & \delta^3 \frac{L_\ell}{8}\po 1+ e^{L\delta^2} \frac{\delta^2}{12}\pf Z^2 A_\ell\,.
\end{eqnarray*}
Gathering the previous bounds,
\begin{multline*}
 |y_\delta-\hat x_\delta| + |\hat v_{\delta} -\tilde  w_\delta  | \ \leqslant \ \frac{\delta}{2}\int_0^\delta |g'(s)-g'(0)|\dd s + \po 1+ \frac{L\delta }{2}\pf|y_\delta-\hat x_\delta|\\
 \leqslant \ \delta^3 \frac{L_\ell}{8}\po 1+ e^{L\delta^2} \frac{\delta^2}{12}\pf Z^2 A_\ell + \frac\delta 2\int_0^\delta \po  e^{L s^2 } \frac{L^2s^4}{24} Z + \frac{L s^2}{2}Z + s|\na U(x)|\pf \dd t\\
  + \po 1+ \frac{L\delta }{2}\pf e^{L \delta^2 } \frac{L\delta^3}{6} Z \\
   \leqslant \ \delta^3 \frac{L_\ell}{8}\po 1+ e^{L\delta^2} \frac{\delta^2}{12}\pf Z^2 A_\ell  + \frac{\delta^3}4|\na U(x)| 
   +   \co        1 + \delta \frac{L+1}{2}  + \delta^3 \frac{L}{40}    \cf e^{L \delta^2 } \frac{L\delta^3}{6} Z \,.
\end{multline*}
\end{proof}

Recall that $\pi_1$ denotes the first marginal of $\pi$, namely the probability law on $\R^d$ with  density $\propto \exp(-U)$. Similarly, denote  by $\pi_2$ the second marginal of $\pi$, which is simply the standard Gaussian distribution on $\R^d$.
  
  \begin{lem}\label{lem:moments}
Under Assumption~\refsmooth\ and \refconvex, for all $\beta\geqslant 0$,
\[\int_{\R^{2d}} |\na U(x)|^\beta \pi_1(\dd x) \ \leqslant \ L^\beta\int_{\R^{2d}} |x-x_\star|^\beta \pi_1(\dd x) \ \leqslant \ L^\beta\po \frac{d+(\beta-2)_+}{m}\pf^{\beta/2}\]
and
\[\int_{\R^{2d}} |v|^\beta \pi_2(\dd v) \ \leqslant \ \po  d+(\beta-2)_+ \pf^{\beta/2}\,.\]
  \end{lem}
    \begin{proof}
 Consider $\mathcal L = -\na U + \Delta$ the generator of the overdamped Langevin diffusion, which leaves invariant $\pi_1$. For $\zeta \geqslant 1$   and $\varphi(x) = |x|^{2\zeta}$,
 \begin{eqnarray*}
\mathcal L \varphi(x) &=& -2\zeta |x|^{2(\zeta-1)}\na U(x) \cdot x + |x|^{2(\zeta-1)}\po 2\zeta d+4\zeta(\zeta-1)\pf \\
& \leqslant & -2m\zeta|x|^{2\zeta} +  |x|^{2(\zeta-1)}\po 2\zeta d+4\zeta(\zeta-1)\pf \,.
 \end{eqnarray*}
 The invariance of $\pi_1$ implies that $\mathbb E_{\pi_1}(\mathcal L\varphi)= 0$, so that
 \begin{eqnarray*}
0 & \leqslant  & -2m\zeta\mathbb E_{\pi_1}\po \varphi\pf +  \po 2\zeta  d+4\zeta(\zeta-1)\pf \mathbb E_{\pi_1} \po \varphi^{(\zeta-1)/\zeta}\pf \\
& \leqslant  & -2m \zeta \mathbb E_{\pi_1}\po \varphi\pf +  \po 2\zeta d+4\zeta(\zeta-1)\pf \po \mathbb E_{\pi_1} \po \varphi\pf \pf^{(\zeta-1)/\zeta}
 \end{eqnarray*}
 from Jensen's inequality, and thus
 \[\po \mathbb E_{\pi_1} \po \varphi\pf \pf^{1/\zeta} \ \leqslant \ \frac{d+2(\zeta-1)}{m}\,.\]
 This concludes the case of $\beta = 2\zeta\geqslant 2$. The case $\beta\in[0,2)$ is obtained by Jensen's inequality. Applying this result to the potential $U(x) = |x|^2/2$ (so that $m=1$) bounds the moments for the velocity.
  \end{proof}
  
  We can now gather the results of the last four lemmas to get the following.
  
  \begin{prop}\label{prop:WppiPV} \

\begin{enumerate}
\item Under Assumptions~\refsmooth\ and \refconvex, for all $p\geqslant 1$,
\begin{eqnarray*}
\W_p\po \pi,\pi P_V\pf  & \leqslant & \delta^2 K_3 \sqrt{d+(p-2)_+}  \\
\| \pi- \pi P_V \|_{TV} & \leqslant &\delta^2   K_4 d
\end{eqnarray*}
with
\begin{eqnarray*}
K_3 &=&L   \co 1 +  e^{L \delta^2 } \po \frac{\delta}6    + \frac{\delta^2 L}{24}\pf  \cf \po 1+  \frac{\delta L}{2\sqrt m}\pf\\
K_4 & = &      L\po 1+\delta \frac14+\delta^2\frac{L}4\pf  \po 1 +\delta \frac{L}{\sqrt m} + \frac{\delta^2L^2}{4m}\pf + \delta  \frac{L^2}m\,.
\end{eqnarray*}
\item  Under Assumptions~\refsmooth, \refconvex\ and \refordren, for all $p\geqslant 1$,
\begin{eqnarray*}
\W_p\po \pi,\pi P_V\pf  & \leqslant & \delta^3  \po K_5 \sqrt{d+(p-2)_+} + L_\ell K_6 \po d + \ell p-2\pf^{\ell/2}\pf\\
\| \pi- \pi P_V \|_{TV} & \leqslant & \delta^3  \po K_7 d + L_\ell K_8 \po d + \ell-1\pf^{(\ell+1)/2}\pf
\end{eqnarray*}
with
\begin{eqnarray*}
K_5 &=&  \frac{L}{4\sqrt m}  
   +   \frac{L }{6} \co        1 + \delta \frac{L+1}{2}  + \delta^3 \frac{L}{40}    \cf e^{L \delta^2 }  \po 1+ \delta \frac{L}{2\sqrt m} \pf  \\
   K_6 & = & 2^{\ell-2}\po 1+ e^{L\delta^2} \frac{\delta^2}{12}\pf \po 1+ \delta + e^{L\delta^2} \frac{\delta^3}{6}\pf^{\ell-2} \po 1+ \frac{\delta L}{2} \pf^{\ell}  \po 1 + \frac{1}{m^{\ell/2}}\pf\\
   K_7 & = & L(1+\delta L) \po 1+\frac{\delta^2L^2}{4m}\pf\\
   K_8 & = &2^{\ell+1}\po 1+ \delta + \frac{\delta^2 L}{2}\pf^{\ell -2} \po 1+ \frac{\delta L}{2}\pf^3\po 1 + \frac{1}{m^{(\ell+1)/2}}\pf\,.
\end{eqnarray*}
\end{enumerate}  

  \end{prop}

\begin{proof}
Both cases are similar so we only detail the first one. For $h\geqslant 1$, denote
\[m_{x,h} \ = \ \int_{\R^2d} |\na U(x)|^h \pi_1(\dd x)\qquad \text{and} \qquad m_{v,h} \ = \ \int_{\R^2d} |v|^h \pi_2(\dd v)\,.\]
For $p\geqslant 1$, using Lemmas~\ref{lem:piPV} and \ref{lem:VerletDiscretization} and the invariance of $\pi$ by $Q_\delta$,
\begin{eqnarray*}
\W_p\po \pi,\pi P_V\pf & = & \W_p\po \pi Q_\delta,\pi P_V\pf \\
& \leqslant &  L \delta^2 \co 1 +  e^{L \delta^2 } \po \delta /6    + \delta^2 L/24\pf  \cf \po \mathbb E_\pi \co \po |V|+\delta/2|\na U(X)|\pf^p\cf\pf^{1/p} \\
& \leqslant &  L \delta^2 \co 1 +  e^{L \delta^2 } \po \delta /6    + \delta^2 L/24\pf  \cf  \co \po m_{v,p}\pf^{1/p} +\delta/2   \po m_{x,p}\pf^{1/p} \cf\\
& \leqslant &    \delta^2 K_3  \sqrt{d+(p-2)_+} \,,
\end{eqnarray*}
where we used Lemma~\ref{lem:moments}. Similarly, from Lemmas~\ref{lem:piPV}, \ref{lem:alpha} and \ref{lem:moments},
\begin{eqnarray*}
\| \pi- \pi P_V \|_{TV} & \leqslant & \delta^3 \frac14 \mathbb E_\pi\po |\na U(X)|^2\pf   +  \delta^2 L\po 1+\delta \frac14+\delta^2\frac{L}4\pf \mathbb E_\pi \po  \po |V|+\delta/2|\na U(X)|\pf^2\pf\\
& = & \delta^3 \frac14 m_{x,2}  +  \delta^2 L\po 1+\delta \frac14+\delta^2\frac{L}4\pf  \po m_{v,2} +\delta m_{v,1}m_{x,1} + \frac{\delta^2}{4}m_{x,2}\pf\ \leqslant \  \delta^2 d K_4\,.
\end{eqnarray*}

Remark that in the computations of the second case, to simplify the last term involved in the bound on $|\Phi_V  - \Phi_{HD} |$ in Lemma~\ref{lem:VerletDiscretization}, we use that
\begin{multline*}
A_\ell' \ := \ 1+\co |X-x_\star|  + \po \delta + e^{L\delta^2} \frac{\delta^3}{6}\pf \po |V|+\frac{\delta}2 |\na U(X)|\pf  \cf^{\ell-2}\\
\leqslant \ 1+\po 1+ \delta + e^{L\delta^2} \frac{\delta^3}{6}\pf^{\ell-2} \co |X-x_\star| \po 1+ \frac{\delta L}{2} \pf +   |V|  \cf^{\ell-2} \\
\leqslant \  \po 1+ \delta + e^{L\delta^2} \frac{\delta^3}{6}\pf^{\ell-2} \po 1+ \frac{\delta L}{2} \pf^{\ell-2} 2^{\ell-3}   \po 2+ |X-x_\star|^{\ell-2} +   |V| ^{\ell-2} \pf
\end{multline*}
and then
\[\po |V|+\frac{\delta}2 |\na U(X)|\pf ^2 A_\ell' \ \leqslant \ 2 \po 1 + \frac{\delta L}{2}\pf^2 \po |V|^2 + |X-x_\star|^2\pf A_\ell' \,. \]
Using then Lemma~\ref{lem:moments} to bound
\begin{multline*}
\po \mathbb E_\pi \co \po |V|^2 + |X-x_\star|^2\pf^p \po 2+ |X-x_\star|^{\ell-2} +   |V| ^{\ell-2} \pf^p \cf\pf^{1/p}  \\
 \leqslant \ 
2 \po \mathbb E \po |V|^{2p}\pf\pf^{1/p} + 2 \po \mathbb E \po |X-x_\star|^{2p}\pf\pf^{1/p}  + 3  \po \mathbb E \po |V|^{\ell p}\pf\pf^{1/p} + 3  \po \mathbb E \po |X-x_\star|^{\ell p}\pf\pf^{1/p} \\
\leqslant \ 5 \po d + \ell p-2\pf^{\ell/2} \po 1 + \frac{1}{m^{\ell/2}}\pf 
\end{multline*}
 yields the claimed expression for $K_6$. Similarly, to bound the last term given in Lemma~\ref{lem:alpha}, we use that
 \begin{multline*}
  1   + \po |X-x_\star|+\delta \po |V| + \frac\delta2|\na U(X)|\pf \pf^{\ell-2}  \ \leqslant \ \po 1+ \delta + \frac{\delta^2 L}{2}\pf^{\ell -2} \po 1 + \po |V|+|X-x_\star|\pf^{\ell-2}\pf\\
   \ \leqslant \ 2^{\ell-3}\po 1+ \delta + \frac{\delta^2 L}{2}\pf^{\ell -2} \po 2 +   |V|^{\ell-2}+|X-x_\star|^{\ell-2}\pf,
 \end{multline*}
 then
 \[\po |V|+ \frac\delta2|\na U(X)|\pf^3 \ \leqslant \ 4\po 1+ \frac{\delta L}{2}\pf^3 \po |V|^3+|X-x_\star|^3\pf \,,\]
 and proceed similarly to $K_6$ to get the expression of $K_8$.
 
\end{proof}

The case of a separable target ensues from the following result.

\begin{lem}\label{lem:tensor}
Let $p\geqslant 2$  and, for $i\in\cco 1,d\ccf$, let $\nu_i,\mu_i \in\mathcal P_p(\R)$ and $\nu = \otimes_{i=1}^d \nu_i$ and $\mu = \otimes_{i=1}^d \mu_i$. Then
\begin{eqnarray*}
\W_p^2(\nu,\mu) & \leqslant & \sum_{i=1}^d \W_{p}^2 \po \nu_i,\mu_i\pf \\
\|\nu - \mu\|_{TV}&\leqslant & \sum_{i=1}^d \|\nu_i - \mu_i\|_{TV}\,.
\end{eqnarray*}
\end{lem}
\begin{proof}
For all $i\in\cco 1,d\ccf$, let $(X_i,Y_i)$ optimal $\W_p$ coupling of $\nu_i$ and $\mu_i$, with $(X_i,Y_i)$ independent from $(X_j,Y_j)$ for $j\neq i$. Then $X=(X_i,\dots,X_d)$ and $Y=(Y_i,\dots,Y_d)$ are respectively distributed according to $\nu$ and $\mu$ so that, for $p\geqslant 2$,
\begin{eqnarray*}
\W_p^2(\nu,\mu) & \leqslant  & \po \mathbb E\co \po \sum_{i=1}^d |X_i-Y_i|^2\pf^{p/2}\cf\pf^{2/p} \ \leqslant  \ \sum_{i=1}^d \po \mathbb E\co   |X_i-Y_i|^p\cf\pf^{2/p} \ = \ \sum_{i=1}^d \W_p^2 \po \nu_i,\mu_i\pf\,. 
\end{eqnarray*}
Alternatively, if  the $(X_i,Y_i)$'s are   optimal  couplings of $\nu_i$ and $\mu_i$ for the total variation distance, then
\begin{eqnarray*}
\|\nu-\mu\|_{TV} & \leqslant & 2\mathbb P \po \exists i\in\cco 1,d\ccf,\ X_i\neq Y_i\pf \ \leqslant \ 2\sum_{i=1}^d \mathbb P \po X_i\neq Y_i\pf \ = \ \sum_{i=1}^d \|\nu_i-\mu_i\|_{TV}\,.
\end{eqnarray*}
\end{proof}

As a conclusion of the analysis of the equilibrium bias:

\begin{proof}[Proof of Propositions~\ref{prop_main:erreur1} and \ref{prop_main:biais}]
Simply combine   Propositions~\ref{prop_main:nupi1} and \ref{prop:WppiPV}. More precisely, for the total variation, to prove Proposition~\ref{prop_main:erreur1}, apply Proposition~\ref{prop_main:nupi1} with
\[n \ := \ \err{1+\left\lceil \frac{ |\ln\po  \delta^{3}d\pf| }{|\ln(1-\delta \kappa)|}\right\rceil \ \leqslant \ 2+\frac{|\ln\po  \delta^{3}d\pf| }{\delta \kappa}}\,,\]
so that
\[(1-\delta \kappa)^{\err{(n-1)}/2} \ \leqslant \ \delta^{3/2} \sqrt d\,.\]
Similarly, to prove Proposition~\ref{prop_main:biais}, apply Proposition~\ref{prop_main:nupi1} with
\[n \ := \ \err{1+\left\lceil \frac{ |\ln\po   \delta^{3/2} \sqrt {d+\ell-1}\pf |}{|\ln(1-\delta \kappa)|}\right\rceil \ \leqslant \ 2+\frac{ |\ln\po   \delta^{3/2} \sqrt {d+\ell-1}\pf | }{\delta \kappa}}\,,\]
so that
\[(1-\delta \kappa)^{\err{(n-1)}/2} \ \leqslant \ \delta^{3/2} \sqrt {d+\ell-1}\,.\]

Finally, under the additional condition  \refind, notice that  necessarily each of the potentials $U_i$ for $i\in\cco 1,d\ccf$ satisfies the conditions \refsmooth\ and \refconvex\ with the same constants $L$ and $m$ than $U$. Moreover, the $\W_p$ and total variation distances are unchanged by an orthonormal change of coordinates, so without loss of generality we suppose that $U(x) = \sum_{i=1}^d U_i(x_i)$.  In particular, $\pi = \otimes_{i=1}^d \nu_i$ where $\nu_i$ is the law of $(X_i,V_i)$ where $(X,V)\sim \pi$. Moreover, the coordinates of the OBABO chain are independent one-dimensional OBABO chains associated to the potentials $U_i$, $i\in\cco 1,d\ccf$, and thus similarly $\pi_\delta = \otimes_{i=1}^d \nu_{i,\delta}$ where $\nu_{i,\delta}$ is the equilibrium of the one-dimensional OBABO chain with potential $U_i$.   The first part of Proposition~\ref{prop_main:biais} applies to $\nu_{i,\delta}$ (with $d=1$) for all $i\in\cco 1,d\ccf$, which together with Lemma~\ref{lem:tensor}  concludes.
\end{proof}

\section{Analysis of the Metropolis-adjusted algorithm}\label{Sec:proofMOBABO}

The results on the OBABO transition $P$ are transfered to the OM(BAB)O transition $P_M$ thanks to the following lemma.

\begin{lem}\label{lem:TVDoeblin}
For all $\nu,\mu\in\mathcal P(\R^{2d})$ and $n\in\N$,
\begin{eqnarray*}
\|\nu P_M^n - \mu P_M^n \|_{TV} & \leqslant &  \|\nu P^n - \mu P^n \|_{TV} + 2\sum_{k=0}^{n-1} (\nu+\mu) P^kP_O(1-\alpha)\,. 
\end{eqnarray*}
Moreover, for  $\nu\in\mathcal P(\R^{2d})$, let $(Z_k)_{k\in\N}$ and $(\tilde Z_k)_{k\in\N}$ be respectively an OBABO and OM(BAB)O chain with initial condition $Z_0=\tilde Z_0 \sim \nu$. Then, for all $u\geqslant 0$ and all $\varphi \in L^1(\pi)$,
\begin{eqnarray*}
 \mathbb{P}\po \left|\frac1n\sum_{k=1}^n   \phi(\tilde Z_k) - \pi(\phi)\right| \geqslant u \pf & \leqslant &    \mathbb{P}\po \left|\frac1n\sum_{k=1}^n   \phi( Z_k) - \pi(\phi)\right| \geqslant u \pf  + \sum_{k=0}^{n-1}  \nu P^{k}P_O(1-\alpha) \,.
\end{eqnarray*}
\end{lem}
\begin{proof}
For the first claim, simply bound
\begin{eqnarray*}
\|\nu P_M^n - \mu P_M^n \|_{TV} & \leqslant & \|\nu P^n - \mu P^n \|_{TV} + \|\nu P_M^n - \nu P^n \|_{TV} + \|\mu P^n - \mu P_M^n \|_{TV}\,.
\end{eqnarray*}
  The two last terms are similar, we bound
\begin{eqnarray*}
\|\nu P_M^n - \nu P^n \|_{TV} & \leqslant & \sum_{k=1}^n \|\nu P^k P_M^{n-k} - \nu P^{k-1}P^{n-k+1}_M \|_{TV}\\
& \leqslant & \sum_{k=1}^n \|\nu P^{k-1} (P-P_M) \|_{TV}\\
& \leqslant & \sum_{k=1}^n \|\nu P^{k-1} P_O(P_V-P_{MV}) \|_{TV}
\end{eqnarray*}
Conclusion follows from Lemma~\ref{lem:piPV}.

The second claim is obtained simply by saying that an OBABO and a OM(BAB)O chains starting at the same point and sampled with the same random variables remain equal up to the first rejection of the OM(BAB)O chain. More precisely, let $(G_n,G_n',W_n)_{n\in\N}$ be an i.i.d. sequence of random variables such that, for all $n\in\N$, $G_n$, $G_n'$ and $W_n$ are independent, $G_n,G_n'\sim \mathcal N(0,I_d)$ and $W_n$ is uniformly distributed over $[0,1]$. Let $Z_0 \sim \nu$ be independent from these variables. We construct $(Z_k)_{k\in \N}$ and $(\tilde Z_k)_{k\in\N}$ as the OBABO and OM(BAB)O chains with initial condition $Z_0$ and  whose transitions at step $n$ are given respectively   in Section~\ref{Sec:OBABO} with variables $G_n,G_n'$ or in Section~\ref{Sec:MOBABO} with variables $G_n,G_n',W_n$. In other words, the damping parts O in both chains use
 the same Gaussian variables, and the only difference between $Z=(X,V)$ and $\tilde Z=(\tilde X,\tilde V)$ is that the proposal given by the Verlet step BAB is always accepted for $Z$ and is accepted iff $W_n \leqslant \alpha(\tilde Z_n')$ for $\tilde Z$, where $\tilde Z_n' = (\tilde X_n,\eta \tilde V_n + \sqrt{1-\eta^2} G_n)$. In particular, the two chains are equal up to the first rejection, in other words, for all $n\geqslant 1$,
 \begin{eqnarray*}
 \mathbb P \po Z_k = \tilde Z_k\ \forall k\in\cco 1,n\ccf\pf & = & \mathbb P \po W_k \leqslant \alpha\po X_k,\eta   V_k + \sqrt{1-\eta^2} G_k\pf \ \forall k\in \cco 0,n-1\ccf\pf\,.
 \end{eqnarray*}
 Bounding the probability of the union by the sum of the probabilities and using that $( X_k,\eta   V_k + \sqrt{1-\eta^2} G_k) \sim \nu P^k P_O$ for all $k\in\N$, we get that
 \begin{eqnarray*}
A\ :=\ \mathbb P \po \exists k\in \cco 1,n\ccf\,, (Z_k) \neq (\tilde Z_k)\pf & \leqslant & \sum_{k=0}^{n-1}  \nu P^{k}P_O(1-\alpha) \,,
\end{eqnarray*}
and for $u\geqslant 0$ we bound
\begin{eqnarray*}
 \mathbb{P}\po \left|\frac1n\sum_{k=1}^n   \phi(\tilde Z_k) - \pi(\phi)\right| \geqslant u \pf & \leqslant &  A +  \mathbb{P}\po \left|\frac1n\sum_{k=1}^n   \phi( Z_k) - \pi(\phi)\right| \geqslant u \pf \,.
\end{eqnarray*}
\end{proof}

In view of this result and of Lemma~\ref{lem:alpha}, in order to prove Proposition~\ref{prop_main:Metropolis} and Theorem~\ref{thm_main:concentration_Metropolis}, it only remains to establish moment estimates for $\nu P^k P_O$ for $\nu\in\mathcal P_p(\R^{2d})$, which is the subject of the following lemma. In the rest of this section,   denote $z_\star = (x_\star,0)$ and $\varphi_\star(z) = |z-z_{\star}|$ for $z\in\R^{2d}$.

\begin{prop}\label{prop:DoeblinPM}
Under the conditions of Proposition~\ref{prop:contraction}, for all $p\geqslant 2$,  $\nu  \in\mathcal P_p(\R^{2d})$ and $n\in \N$, 
\begin{equation*}
\po \nu P^nP_O\po \varphi_\star^p\pf\pf^{1/p} \ \leqslant \    K_1 (1-\delta\kappa)^{n/2} \po  \nu   (\varphi_\star^p)\pf^{1/p} +   ( K_1 +3)\po 1 + \frac{1}{\sqrt m} + \frac{2\delta K_1K_3}{\kappa}\pf \sqrt{d+p-2}\,.
\end{equation*}
\end{prop}
\begin{proof}
Using that $P_O$ does not increase the $\W_p$ distance (see the proof of Proposition~\ref{prop_main:nupi1}) and Theorem~\ref{thm_main:Wpcontract}
\begin{eqnarray*}
 \po  \nu P^n P_O (\varphi_\star^p)\pf^{1/p} & \leqslant &   \W_p (\nu P^nP_O,\pi_\delta P^nP_O) + \po  \pi_\delta P^nP_O(\varphi_\star^p)\pf^{1/p}\\
  & \leqslant &  K_1 (1-\delta\kappa)^{n/2}  \W_p(\nu ,\pi_\delta) + \po  \pi_\delta P_O(\varphi_\star^p)\pf^{1/p}\\
  & \leqslant &  K_1 (1-\delta\kappa)^{n/2} \po  \nu   (\varphi_\star^p)\pf^{1/p} + K_1 \po  \pi_\delta   (\varphi_\star^p)\pf^{1/p} +\po  \pi_\delta P_O(\varphi_\star^p)\pf^{1/p}\,.
\end{eqnarray*}
Let $(X,V)\sim \pi_\delta$ and, independently, $G\sim \mathcal N(0,I_d)$, so that $(X,\eta V  + \sqrt{1-\eta^2} G)\sim \pi_\delta P_O$. For  $p\geqslant 1$,
\begin{eqnarray*}
\po  \pi_\delta P_O(\varphi_\star^p)\pf^{1/p} & = & \po \mathbb E \po |(X,\eta V  + \sqrt{1-\eta^2} G)-z_\star|^p\pf \pf^{1/p}\\
& \leqslant &  \po \mathbb E \po |(X,\eta V)-z_\star|^p\pf \pf^{1/p} + \sqrt{1-\eta^2}\po \mathbb E \po | G|^p\pf \pf^{1/p}\\
& \leqslant &  \po \mathbb E \po |(X, V)-z_\star|^p\pf \pf^{1/p} + \sqrt{d+p-2}\,,
\end{eqnarray*}
so that
\begin{eqnarray*}
  \po \nu P^nP_O (\varphi_\star^p) \pf^{1/p} & \leqslant & K_1 (1-\delta\kappa)^{n/2} \po  \nu   (\varphi_\star^p)\pf^{1/p} +( K_1 +1)\po  \pi_\delta   (\varphi_\star^p)\pf^{1/p}  +   \sqrt{d+p-2}\,.
\end{eqnarray*}
The moments of $\pi_\delta$ are bounded by combining Lemma~\ref{lem:moments} and Proposition~\ref{prop_main:erreur1}, as
\begin{eqnarray*}
\po  \pi_\delta   (\varphi_\star^p)\pf^{1/p} & \leqslant  & \po  \pi    (\varphi_\star^p)\pf^{1/p} + \W_p\po \pi,\pi_\delta\pf\\
& \leqslant & \po 1 + \frac{1}{\sqrt m} + \frac{2\delta K_1K_3}{\kappa}\pf \sqrt{d+p-2}\,.
\end{eqnarray*}

\end{proof}

\begin{cor}\label{cor:nuPkPO}
Under the conditions of Proposition~\ref{prop:contraction}:
\begin{enumerate}
\item For all  $\nu  \in\mathcal P_2(\R^{2d})$ and $n\in \N$, 
\begin{equation*}
\sum_{k=0}^{n-1} \nu P^kP_O(1-\alpha) \ \leqslant \ K_9 \po \delta  \nu   (\varphi_\star^2)   + \delta^2 n d \pf 
\end{equation*}
with, considering $K_1$ and $K_3$ as in Theorem~\ref{thm_main:Wpcontract} and  Proposition~\ref{prop:WppiPV},
\begin{multline*}
K_9 \ =\ 2 L ( K_1 +3)^2\po 1 + \frac{1}{\sqrt m} + \frac{2\delta K_1K_3}{\kappa}\pf^2 \po 1 + \frac{1}{\kappa}\pf \\
\ \times \ \co \frac{\delta L}{8} + \po 1+ \frac \delta 4 + \frac{\delta^2L}4\pf \po 1+ \delta^2 \frac{L^2}{4}\pf \cf\,.   
\end{multline*}
\item If \refordren\ holds then for all $\nu  \in\mathcal P_{\ell+1}(\R^{2d})$ and $n\in \N$, 
\begin{equation*}
\sum_{k=0}^{n-1} \nu P^kP_O(1-\alpha) \ \leqslant \ K_{10} \po \delta^2  \nu   (\varphi_\star^{\ell+1})   + \delta^3 n (d+\ell-1)^{(\ell+1)/2} \pf 
\end{equation*}
with, considering $K_1$ and $K_3$ as in Theorem~\ref{thm_main:Wpcontract} and  Proposition~\ref{prop:WppiPV},
\begin{multline*}
K_{10} \ =\ \co 1+  2^{\ell} ( K_1 +3)^{\ell+1}\po 1 + \frac{1}{\sqrt m} + \frac{2\delta K_1K_3}{\kappa}\pf^{\ell+1} \po \frac{1}{\kappa}+1\pf\cf\\
\times \  \co      \frac{L+2L^2}{4}    \po 1+ \delta^2 \frac{L^2}{4}\pf   +  2^{(\ell-3)/2}L_\ell  \po 1+\frac{L\delta}{2}\pf^{3} \po 1+\delta+\frac{\delta^2L}{2}\pf^{\ell-2}  \cf \,.
\end{multline*}
\item If \refind\ holds then for all $\nu  \in\mathcal P_{\ell+1}(\R^{2d})$ and $n\in \N$, denoting $\varphi_{\star,i}(z) = |\mathcal Q^{-1}(z_i-z_{\star,i})|$ for $z\in\R^{2d}$,
\begin{equation*}
\sum_{k=0}^{n-1} \nu P^kP_O(1-\alpha) \ \leqslant \ K_{10} \po \delta^2  \sum_{i=1}^d \nu    (\varphi_{\star,i}^{\ell+1})   + \delta^3 n d \ell^{(\ell+1)/2} \pf 
\end{equation*}
with $K_{10}$ given above.
\end{enumerate}
\end{cor}

\begin{proof}
The first two claims are straightforward consequences of the control of $1-\alpha$ given by Lemma~\ref{lem:alpha} and of the moment estimates of Proposition~\ref{prop:DoeblinPM}. Indeed,   under Assumptions~\refsmooth\ and \refconvex, $|\na U(x)|\leqslant L|x-x_\star|$ for all $x\in\R^d$, so that Lemma~\ref{lem:alpha} implies 
  \begin{equation}\label{eq:bound_alpha_rough}
  1-\alpha(x,v) \ \leqslant \   \delta^2 L\co \frac{\delta L}{8} + \po 1+ \frac \delta 4 + \frac{\delta^2L}4\pf \po 1+ \delta^2 \frac{L^2}{4}\pf \cf\po |v|^2 + |x-x_\star|^2\pf \,.
  \end{equation}
 Similarly, under the additional condition \refordren, using that
 \[ 1   + \po |x-x_\star|+\delta \po |v| + \frac\delta2|\na U(x)|\pf \pf^{\ell-2}  \ \leqslant \ \po 1+ \delta + \frac{\delta^2 L}{2}\pf^{\ell -2} \po 1 + \po |v|+|x-x_\star|\pf^{\ell-2}\pf\]
 and that
 \[\po |v|+ \frac\delta2|\na U(x)|\pf^3 \ \leqslant \  \po 1+ \frac{\delta L}{2}\pf^3 \po |v| +|x-x_\star|\pf^3 \,,\]
 we see that Lemma~\ref{lem:alpha} implies  
    \begin{multline}
  1-\alpha(x,v) \ \leqslant \    \delta^3 L\co \frac{  L}{8} +     \frac{1+L} 4    \po 1+ \delta^2 \frac{L^2}{4}\pf \cf\po |v|^2 + |x-x_\star|^2\pf  \\
  +   \frac{\delta^3}8 L_\ell \po 1+\frac{L\delta}{2}\pf^{3} \po 1+\delta+\frac{\delta^2L}{2}\pf^{\ell-2} \po \po   |x-x_\star|+|v|\pf^3  + \po   |x-x_\star|+|v|\pf^{\ell+1}\pf \\
  \ \leqslant \  \delta^3  \co      \frac{L+2L^2}{4}    \po 1+ \delta^2 \frac{L^2}{4}\pf   +  \frac{L_\ell}8 \po 1+\frac{L\delta}{2}\pf^{3} \po 1+\delta+\frac{\delta^2L}{2}\pf^{\ell-2} \po 2^{3/2}+2^{(\ell+1)/2}\pf\cf  \\
  \ \times \ \po 1 +    |z-z_\star|^{\ell+1}\pf \label{eq:bound_alpha_rough_ordren}
  \end{multline}
  where we bounded $(|x-x_\star|+|v|)^k$ by $2^{k/2}|z-z_\star|^k$ for $k=3$ and $\ell+1$ and $|z-z_\star|^k$ by $1+|z-z_\star|^{\ell +1}$ for $k = 2$ and $3$. Moreover, using that $(s+t)^p \leqslant 2^{p-1}(s^p+t^p)$ for $s,t>0$, we see that Proposition~\ref{prop:DoeblinPM} implies that for all $p\geqslant 2$,
\[ \nu P^nP_O\po \varphi_\star^p\pf \ \leqslant \   2^{p-1} K_1^p (1-\delta\kappa)^{np/2}  \nu   (\varphi_\star^p) +  2^{p-1} ( K_1 +3)^p\po 1 + \frac{1}{\sqrt m} + \frac{2\delta K_1K_3}{\kappa}\pf^p (d+p-2)^{p/2}\,,\]
so that
\[
\sum_{k=0}^{n-1} \nu P^kP_O\po \varphi_\star^p\pf \ \leqslant \  2^{p-1} ( K_1 +3)^p\po 1 + \frac{1}{\sqrt m} + \frac{2\delta K_1K_3}{\kappa}\pf^p \po \frac{1}{\delta\kappa } \nu   (\varphi_\star^p) + n (d+p-2)^{p/2}\pf
\]
Combining this with \eqref{eq:bound_alpha_rough} (with $p=2$) gives the expression of $K_9$, and with \eqref{eq:bound_alpha_rough_ordren}  (with $p=\ell+1$) gives $K_{10}$.

In the separable case (condition \refind),  denoting $(y,v)= (\mathcal Q^{-1} x,\mathcal Q^{-1} v)$ and $H_i(y_i,w_i)=U_{i}(y_i)+|w_i|^2/2$ for $i\in\cco 1,d\ccf$, using that $H(x,v) = \sum_{i=1}^d H_i(y_i,w_i)$ and that $\mathcal Q^{-1}  \Phi_V = \Phi_V \mathcal Q^{-1} $, we get that
\[1-\alpha(x,v) \ \leqslant \ 1 - e^{-\sum_{i=1}^d \po H_i\circ \Phi_V (y_i,w_i) - H_i(y_i,w_i)\pf_+} \ := \ 1 - \prod_{i=1}^d \alpha_i(y_i,w_i)\,,\]
where $\alpha_i$ is the acceptance probability of a one-dimensional OM(BAB)O chain with potential $U_i$. Bounding the probability of the union by the sum of the probability we get that
\[1-\alpha(x,v) \ \leqslant \ \sum_{i=1}^d \po 1 - \alpha_i(y_i,w_i)\pf\,.\]
Moreover, the OBABO transition $P$ is such that, up to the orthonormal change of variables  $(y,w)=(\mathcal Q^{-1} x,\mathcal Q^{-1} v)$, the coordinates (position/velocities) of the chain are independent one-dimensional OBABO chains. Similarly the transition $P_O$ preserves this independence. Applying the second claim of Corollary~\ref{cor:nuPkPO} (with $d=1$) to the $i^{th}$ one-dimensional chain gives 
\begin{equation*}
\sum_{k=0}^{n-1} \nu P^kP_O(1-\alpha_i) \ \leqslant \ K_{10} \po \delta^2   \nu    (\varphi_{\star,i}^{\ell+1})   + \delta^3 n   \ell^{(\ell+1)/2} \pf 
\end{equation*}
for $i\in\cco 1,d\ccf$. Summing these $d$ inequalities concludes.

\end{proof}

\begin{proof}[Proof of Proposition~\ref{prop_main:Metropolis}]
Simply combine Lemma~\ref{lem:TVDoeblin} and  Corollary~\ref{cor:nuPkPO}.

\end{proof}

\begin{proof}[Proof of Theorem~\ref{thm_main:concentration_Metropolis}]
This follows from the concentration bound for the unadjusted chain (Theorem~\ref{thm_main:concentration}), the bound between the adjusted and unadjusted process obtained by combining Lemma~\ref{lem:TVDoeblin} and  Corollary~\ref{cor:nuPkPO}, and the dual representation of the $\W_1$ distance (see e.g. \cite{VillaniOldNew}) that implies that $|\pi(\phi)-\pi_\delta(\phi)|\leqslant\W_1(\pi,\pi_\delta)$ for all $1$-Lipschitz function $\phi$.
\end{proof}

\section*{Acknowledgments}

P. Monmarch\'e thanks Tony Leli\`evre, Benedict Leimkuhler, Gabriel Stoltz and Alaind Durmus for  stimulating discussions on the topic of Langevin integrators. He acknowledges financial support from the French ANR grant EFI (Entropy, flows, inequalities, ANR-17-CE40-0030).

\bibliographystyle{plain}
\bibliography{biblio}

\begin{thebibliography}{10}

\bibitem{Bach}
Francis {Bach}.
\newblock {On the Effectiveness of Richardson Extrapolation in Machine
  Learning}.
\newblock {\em arXiv e-prints}, page arXiv:2002.02835, February 2020.

\bibitem{BakryGentilLedoux}
Dominique Bakry, Ivan Gentil, and Michel Ledoux.
\newblock {\em Analysis and geometry of {M}arkov diffusion operators}, volume
  348 of {\em Grundlehren der Mathematischen Wissenschaften [Fundamental
  Principles of Mathematical Sciences]}.
\newblock Springer, Cham, 2014.

\bibitem{BierkensFearnheadRoberts}
Joris {Bierkens}, Paul {Fearnhead}, and Gareth {Roberts}.
\newblock The zig-zag process and super-efficient sampling for {B}ayesian
  analysis of big data.
\newblock {\em Ann. Statist.}, 47(3):1288--1320, 2019.

\bibitem{Malrieu}
Fran{\c c}ois Bolley, Arnaud Guillin, and Florent Malrieu.
\newblock Trend to equilibrium and particle approximation for a weakly
  selfconsistent vlasov-fokker-planck equation.
\newblock {\em ESAIM: Mathematical Modelling and Numerical Analysis -
  Mod\'elisation Math\'ematique et Analyse Num\'erique}, 44(5):867--884, 2010.

\bibitem{EberleHMC}
Nawaf Bou-Rabee, Andreas Eberle, and Raphael Zimmer.
\newblock {Coupling and convergence for Hamiltonian Monte Carlo}.
\newblock {\em The Annals of Applied Probability}, 30(3):1209 -- 1250, 2020.

\bibitem{Bourabee2}
Nawaf Bou-Rabee and Eric Vanden-Eijnden.
\newblock Pathwise accuracy and ergodicity of metropolized integrators for
  sdes.
\newblock {\em Communications on Pure and Applied Mathematics}, 63(5):655--696,
  2010.

\bibitem{Bourabee1}
Nawaf Bou-Rabee and Eric Vanden-Eijnden.
\newblock A patch that imparts unconditional stability to explicit integrators
  for langevin-like equations.
\newblock {\em Journal of Computational Physics}, 231(6):2565 -- 2580, 2012.

\bibitem{BussiParrinello}
Giovanni Bussi and Michele Parrinello.
\newblock Accurate sampling using langevin dynamics.
\newblock {\em Phys. Rev. E}, 75:056707, May 2007.

\bibitem{ChafaiMalrieuParoux}
Djalil Chafa{\"{\i}}, Florent Malrieu, and Katy Paroux.
\newblock On the long time behavior of the {TCP} window size process.
\newblock {\em Stochastic Process. Appl.}, 120(8):1518--1534, 2010.

\bibitem{Dwivedi2}
Yuansi Chen, Raaz Dwivedi, Martin~J. Wainwright, and Bin Yu.
\newblock Fast mixing of metropolized hamiltonian monte carlo: Benefits of
  multi-step gradients.
\newblock {\em Journal of Machine Learning Research}, 21(92):1--72, 2020.

\bibitem{Chatterji2}
Xiang {Cheng}, Niladri~S. {Chatterji}, Yasin {Abbasi-Yadkori}, Peter~L.
  {Bartlett}, and Michael~I. {Jordan}.
\newblock {Sharp Convergence Rates for Langevin Dynamics in the Nonconvex
  Setting}.
\newblock {\em arXiv e-prints}, page arXiv:1805.01648, May 2018.

\bibitem{Chatterji1}
Xiang Cheng, Niladri~S. Chatterji, Peter~L. Bartlett, and Michael~I. Jordan.
\newblock Underdamped langevin {MCMC}: A non-asymptotic analysis.
\newblock In Sébastien Bubeck, Vianney Perchet, and Philippe Rigollet,
  editors, {\em Proceedings of the 31st Conference On Learning Theory},
  volume~75 of {\em Proceedings of Machine Learning Research}, pages 300--323.
  PMLR, 06--09 Jul 2018.

\bibitem{dalalyan2}
A.~Dalalyan, Lionel Riou-Durand, and Avetik~G. Karagulyan.
\newblock Bounding the error of discretized langevin algorithms for
  non-strongly log-concave targets.
\newblock {\em ArXiv}, abs/1906.08530, 2019.

\bibitem{dalalyan1}
Arnak~S. Dalalyan and Lionel Riou-Durand.
\newblock On sampling from a log-concave density using kinetic langevin
  diffusions.
\newblock {\em Bernoulli}, 26(3):1956--1988, 08 2020.

\bibitem{DoucetHMC}
George {Deligiannidis}, Daniel {Paulin}, Alexandre {Bouchard-C{\^o}t{\'e}}, and
  Arnaud {Doucet}.
\newblock {Randomized Hamiltonian Monte Carlo as Scaling Limit of the Bouncy
  Particle Sampler and Dimension-Free Convergence Rates}.
\newblock {\em arXiv e-prints}, page arXiv:1808.04299, August 2018.

\bibitem{Djellout}
Hac{\`e}ne. Djellout, Arnaud Guillin, and Li-Ming Wu.
\newblock Transportation cost-information inequalities and applications to
  random dynamical systems and diffusions.
\newblock {\em Ann. Probab.}, 32(3B):2702--2732, 07 2004.

\bibitem{DMS2009}
Jean Dolbeault, Cl{\'e}ment Mouhot, and Christian Schmeiser.
\newblock Hypocoercivity for kinetic equations with linear relaxation terms.
\newblock {\em C. R. Math. Acad. Sci. Paris}, 347(9-10):511--516, 2009.

\bibitem{durmus2019}
Alain Durmus and Éric Moulines.
\newblock High-dimensional bayesian inference via the unadjusted langevin
  algorithm.
\newblock {\em Bernoulli}, 25(4A):2854--2882, 11 2019.

\bibitem{DurmusRomberg}
Alain Durmus, Umut Simsekli, {\'E}ric Moulines, Roland Badeau, and Gael
  Richard.
\newblock {Stochastic Gradient Richardson-Romberg Markov Chain Monte Carlo}.
\newblock In {\em {Thirtieth Annual Conference on Neural Information Processing
  Systems (NIPS)}}, Barcelone, Spain, December 2016.

\bibitem{Dwivedi}
Raaz Dwivedi, Yuansi Chen, Martin~J. Wainwright, and Bin Yu.
\newblock Log-concave sampling: Metropolis-hastings algorithms are fast.
\newblock {\em Journal of Machine Learning Research}, 20(183):1--42, 2019.

\bibitem{EberleGuillinZimmer}
Andreas Eberle, Arnaud Guillin, and Raphael Zimmer.
\newblock Couplings and quantitative contraction rates for langevin dynamics.
\newblock {\em Ann. Probab.}, 47(4):1982--2010, 07 2019.

\bibitem{MonmarcheGuillin}
Arnaud {Guillin} and Pierre {Monmarch{\'e}}.
\newblock {Uniform long-time and propagation of chaos estimates for mean field
  kinetic particles in non-convex landscapes}.
\newblock {\em arXiv e-prints}, page arXiv:2003.00735, March 2020.

\bibitem{GuillinWang}
Arnaud Guillin and Feng-Yu Wang.
\newblock Degenerate fokker–planck equations: Bismut formula, gradient
  estimate and harnack inequality.
\newblock {\em Journal of Differential Equations}, 253(1):20 -- 40, 2012.

\bibitem{Herau2007}
Frederic H{\'e}rau.
\newblock Short and long time behavior of the {F}okker-{P}lanck equation in a
  confining potential and applications.
\newblock {\em J. Funct. Anal.}, 244(1):95--118, 2007.

\bibitem{HOROWITZ}
Alan~M. Horowitz.
\newblock A generalized guided monte carlo algorithm.
\newblock {\em Physics Letters B}, 268(2):247 -- 252, 1991.

\bibitem{Joulin}
Aldéric Joulin and Yann Ollivier.
\newblock Curvature, concentration and error estimates for markov chain monte
  carlo.
\newblock {\em Ann. Probab.}, 38(6):2418--2442, 11 2010.

\bibitem{Kuwada}
Kazumasa Kuwada.
\newblock Duality on gradient estimates and {W}asserstein controls.
\newblock {\em J. Funct. Anal.}, 258(11):3758--3774, 2010.

\bibitem{LedouxConcentration}
Michel Ledoux.
\newblock The concentration of measure phenomenon.
\newblock {\em Mathematical Surveys and Monographs}, 89, 2001.

\bibitem{idealHMC2}
Yin~Tat Lee and Santosh~S. Vempala.
\newblock Convergence rate of riemannian hamiltonian monte carlo and faster
  polytope volume computation.
\newblock In {\em Proceedings of the 50th Annual ACM SIGACT Symposium on Theory
  of Computing}, STOC 2018, page 1115–1121, New York, NY, USA, 2018.
  Association for Computing Machinery.

\bibitem{Leimkuhler}
Benedict Leimkuhler and Charles Matthews.
\newblock {Rational Construction of Stochastic Numerical Methods for Molecular
  Sampling}.
\newblock {\em Applied Mathematics Research eXpress}, 2013(1):34--56, 06 2012.

\bibitem{LelievreFreeEnergy}
Tony Leli\`evre, Mathias Rousset, and Gabriel Stoltz.
\newblock {\em Free energy computations: A mathematical perspective}.
\newblock Imperial College Press, 2010.

\bibitem{LelievreStoltz}
Tony Lelièvre and Gabriel Stoltz.
\newblock Partial differential equations and stochastic methods in molecular
  dynamics.
\newblock {\em Acta Numerica}, 25:681–880, 2016.

\bibitem{Chatterji3}
Yi-An Ma, Niladri~S. Chatterji, Xiang Cheng, Nicolas Flammarion, Peter~L.
  Bartlett, and Michael~I. Jordan.
\newblock {Is there an analog of Nesterov acceleration for gradient-based
  MCMC?}
\newblock {\em Bernoulli}, 27(3):1942 -- 1992, 2021.

\bibitem{MadrasSezer}
Neal Madras and Deniz Sezer.
\newblock Quantitative bounds for markov chain convergence: Wasserstein and
  total variation distances.
\newblock {\em Bernoulli}, 16(3):882--908, 08 2010.

\bibitem{MalrieuTalay}
Florent {Malrieu} and Denis {Talay}.
\newblock Concentration inequalities for euler schemes.
\newblock {\em Monte Carlo and Quasi-Monte Carlo Methods}, page 355, 2004.

\bibitem{Mangoubi}
Oren {Mangoubi} and Aaron {Smith}.
\newblock {Rapid Mixing of Hamiltonian Monte Carlo on Strongly Log-Concave
  Distributions}.
\newblock {\em arXiv e-prints}, page arXiv:1708.07114, August 2017.

\bibitem{Mangoubi2}
Oren {Mangoubi} and Nisheeth~K. {Vishnoi}.
\newblock {Dimensionally Tight Bounds for Second-Order Hamiltonian Monte
  Carlo}.
\newblock {\em arXiv e-prints}, page arXiv:1802.08898, February 2018.

\bibitem{OldLangevin1}
Bimal Mishra and Tamar Schlick.
\newblock The notion of error in langevin dynamics. i. linear analysis.
\newblock {\em The Journal of Chemical Physics}, 105(1):299--318, 1996.

\bibitem{MoiRomberg}
P.~Monmarch{\'e}.
\newblock {E}fficiency bounds for stochastic gradient {MCMC} with {R}omberg
  interpolation.
\newblock {\em in preparation}.

\bibitem{MonmarcheVFP}
P.~Monmarch{\'e}.
\newblock Long-time behaviour and propagation of chaos for mean field kinetic
  particles.
\newblock {\em Stochastic Process. Appl.}, 127(6):1721--1737, 2017.

\bibitem{MonmarchePDMP}
Pierre {Monmarch{\'e}}.
\newblock {On $\mathcal H^{1}$ and entropic convergence for contractive PDMP}.
\newblock {\em Electronic Journal of Probability}, 20, December 2015.

\bibitem{MonmarcheRecuitHypo}
Pierre {Monmarch{\'e}}.
\newblock Hypocoercivity in metastable settings and kinetic simulated
  annealing.
\newblock {\em Probability Theory and Related Fields}, Jan 2018.

\bibitem{MonmarcheContraction}
Pierre {Monmarch{\'e}}.
\newblock {Almost sure contraction for diffusions on $\mathbb R^d$. Application
  to generalised Langevin diffusions}.
\newblock {\em arXiv e-prints}, page arXiv:2009.10828, September 2020.

\bibitem{Monmarche2019Kinetic_walk}
Pierre {Monmarch{\'e}}.
\newblock {Kinetic walks for sampling}.
\newblock {\em ALEA, Lat. Am. J. Probab. Math. Stat.}, 17:491, 2020.

\bibitem{Neal}
Radford~M. Neal.
\newblock {MCMC} using {Hamiltonian} dynamics.
\newblock {\em Handbook of Markov Chain Monte Carlo}, 54:113--162, 2010.

\bibitem{Ollivier}
Yann Ollivier.
\newblock Ricci curvature of markov chains on metric spaces.
\newblock {\em Journal of Functional Analysis}, 256(3):810 -- 864, 2009.

\bibitem{Pavliotis}
Michela Ottobre, Natesh~S. Pillai, Frank~J. Pinski, and Andrew~M. Stuart.
\newblock A function space hmc algorithm with second order langevin diffusion
  limit.
\newblock {\em Bernoulli}, 22(1):60--106, 02 2016.

\bibitem{QinHobert}
Qian Qin and J.~Hobert.
\newblock On the limitations of single-step drift and minorization in markov
  chain convergence analysis, 2020.

\bibitem{RobertRosenthal}
Gareth~O. Roberts and Jeffrey~S. Rosenthal.
\newblock One-shot coupling for certain stochastic recursive sequences.
\newblock {\em Stochastic Processes and their Applications}, 99(2):195 -- 208,
  2002.

\bibitem{RobertTweedie}
Gareth~O. Roberts and Richard~L. Tweedie.
\newblock Geometric l 2 and l 1 convergence are equivalent for reversible
  markov chains.
\newblock {\em Journal of Applied Probability}, 38(A):37–41, 2001.

\bibitem{SLSCC}
Anthony Scemama, Tony Leli{\`e}vre, Gabriel Stoltz, Eric Canc{\`e}s, and Michel
  Caffarel.
\newblock An efficient sampling algorithm for variational monte carlo.
\newblock {\em J Chem Phys.}, 125(11), 2006.

\bibitem{idealHMC1}
Christof Seiler, Simon Rubinstein-Salzedo, and Susan Holmes.
\newblock Positive curvature and hamiltonian monte carlo.
\newblock In Z.~Ghahramani, M.~Welling, C.~Cortes, N.~D. Lawrence, and K.~Q.
  Weinberger, editors, {\em Advances in Neural Information Processing Systems
  27}, pages 586--594. Curran Associates, Inc., 2014.

\bibitem{Song}
Zexi {Song} and Zhiqiang {Tan}.
\newblock {Hamiltonian Assisted Metropolis Sampling}.
\newblock {\em arXiv e-prints}, page arXiv:2005.08159, May 2020.

\bibitem{Talay}
Denis Talay.
\newblock Stochastic {H}amiltonian systems: exponential convergence to the
  invariant measure, and discretization by the implicit {E}uler scheme.
\newblock {\em Markov Process. Related Fields}, 8(2):163--198, 2002.

\bibitem{OldLangevin2}
Grigori N. MilsteinMichael~V. Tretyakov.
\newblock {\em Stochastic Numerics for Mathematical Physics}.
\newblock Springer, Berlin, Heidelberg, 2004.

\bibitem{Tuckerman}
Mark~E. Tuckerman.
\newblock {\em Statistical mechanics theory and molecular simulation}.
\newblock Oxford University Press, 2010.

\bibitem{Villani2009}
C{\'e}dric Villani.
\newblock Hypocoercivity.
\newblock {\em Mem. Amer. Math. Soc.}, 202(950):iv+141, 2009.

\bibitem{VillaniOldNew}
C{\'e}dric Villani.
\newblock {\em Optimal transport, old and new}, volume 338 of {\em Grundlehren
  der Mathematischen Wissenschaften [Fundamental Principles of Mathematical
  Sciences]}.
\newblock Springer-Verlag, Berlin, 2009.

\bibitem{Zajic}
Tim {Zajic}.
\newblock {Non-asymptotic error bounds for scaled underdamped Langevin MCMC}.
\newblock {\em arXiv e-prints}, page arXiv:1912.03154, December 2019.

\end{thebibliography}

\end{document}